\documentclass[12pt,a4paper,english]{article}

\usepackage{graphicx}
\usepackage{xcolor}
\usepackage{tikz}
\usetikzlibrary{backgrounds}
\usepackage{float}
\floatstyle{boxed} 
\restylefloat{figure}

\newcommand{\Pro}{\ensuremath{\mathbb{P}}}
\newcommand{\pr}{\ensuremath{\mathbb{P}}}
\newcommand{\Var}{\ensuremath{\mathrm{Var}}}

\newcommand{\Po}{\ensuremath{\mathrm{Po}}}
\newcommand{\Ex}{\ensuremath{\mathbb{E}}}

\newcommand{\Bin}{\ensuremath{\mathrm{Bin}}}

\newcommand{\cZ}{\ensuremath{\mathcal{Z}}}

\newcommand{\minus}{\,\text{--}\,}
\newcommand{\cL}{\ensuremath{\mathcal{L}}}
\newcommand{\M}{\ensuremath{{m_p}}}

\newcommand{\s}{\ensuremath{\mathrm{span}}}

\newcommand{\cS}{\ensuremath{\mathcal{S}}}

\newcommand{\ind}{{\mathbb I}}
\newcommand{\eps}{\varepsilon}

\usepackage{enumerate}

\usepackage[toc]{glossaries}

\usepackage{hyperref}
\usepackage{amsmath,amsthm}%
\usepackage{amsfonts}%
\usepackage{amssymb}%
\usepackage{graphicx}
\newtheorem{theorem}{Theorem}
\theoremstyle{plain}

\newtheorem{lemma}[theorem]{Lemma}

\newtheorem{proposition}[theorem]{Proposition}
\newtheorem{remark}[theorem]{Remark}

\begin{document}
\title{The component structure of dense random subgraphs of the hypercube.}
\author{
Colin McDiarmid\\Department of Statistics,\\ University of Oxford,\\ 24 - 29 St Giles',\\ Oxford, OX1 3LB, UK.\\cmcd@stats.ox.ac.uk\\
\and
Alex Scott\thanks{Supported by a Leverhulme Trust Research Fellowship}\\Mathematical Institute,\\ 
University of Oxford,\\ Radcliffe Observatory Quarter,\\  Woodstock Road, \\ Oxford, OX2 6GG, UK.\\scott@maths.ox.ac.uk\\ 
\and 
Paul Withers\\Mathematical Institute,\\ University of Oxford,\\ Radcliffe Observatory Quarter,\\  Woodstock Road, \\ Oxford, OX2 6GG, UK.\\paul.n.withers@gmail.com}
\date{}
%\date{\bfseries\today}
\maketitle
\begin{abstract}
\noindent
Given $p \in (0,1)$, %Fix $0 < p < 1/2$, 
we let $Q_p=  Q_p^d$ be the random subgraph of the $d$-dimensional hypercube $Q^d$ where edges are present independently with probability $p$. 
It is well known that, as $d \rightarrow \infty$, if $p>\frac12$ then with high probability $Q_p$ is connected; and if $p<\frac12$ then with high probability $Q_p$ consists of one giant component together with many smaller components which form the `fragment'.

Here we fix $p \in (0,\frac12)$, and investigate the fragment, and how it sits inside the hypercube. For example, we give asymptotic estimates for the mean numbers of components in the fragment of each size, and describe their asymptotic distributions, much extending earlier work of Weber.
\end{abstract}

\section{Introduction}\label{3sec:intro}

\noindent
The hypercube  $Q=Q^d$ is the graph with vertex set $\{0,1\}^d$ and with two vertices adjacent when they differ in exactly one co-ordinate. Alternatively it can be considered as the graph on the power set of $[d]=\{1,2,\dots ,d\}$ in which two sets are adjacent when their symmetric difference is a singleton.
We consider the random subgraph $Q_p=Q^d_p$ where the edges appear independently with fixed probability $p$, and examine the component structure as $d\rightarrow \infty$. We say that $Q_p$ has a property \emph{with high probability (or whp)} if the property holds with probability tending to 1 as $d\rightarrow \infty$, and $Q_p$ has a property \emph{with very high probability (or wvhp)} if it holds with probability $1-e^{-\Omega(d)}$.

Burtin~\cite{burtin1977} considered random subgraphs in the dense case and showed that, for fixed $p<1/2$, whp $Q_p$ is disconnected and, for fixed $p>1/2$, whp $Q_p$ is connected.  Erd\H{o}s and Spencer \cite{79} showed that for $p=1/2$, $Q_p$ is connected with probability tending to $e^{-1}$ (see also Bollob\'as~\cite[Theorem 14.3]{bollobasbook}).  Also Weber~\cite{Weber1986} considered the dense case -- we will discuss  his work shortly.  Ajtai, Koml\'{o}s and Szemer\'{e}di \cite{72} looked at the sparse case, and demonstrated that a phase transition occurs at $p=1/d$\,: for $p=\lambda/d$ with $\lambda>1$, whp the largest component of $Q_p$ has size $\Omega(2^d)$ and the second largest has size $o(2^d)$, while for $\lambda<1$ whp the largest component has size $o(2^d)$. Bollob\'{a}s, Kohayakawa and \L{}uczak \cite{109,57BB,53BB,54BB} gave more detailed results around the phase transition at $p=1/d$, and investigated the minimum degree, connectedness and the existence of a complete matching in the sequence of subgraphs of $Q^d$ formed by adding edges randomly, one at a time. They showed that, almost surely, this graph process becomes connected exactly at the moment when the last isolated vertex disappears,  and at this time a complete matching emerges. See~\cite{BCHS2006,HN2020} for more recent work concerning behaviour around the phase transition and for further references.

This paper looks at the sizes of the components of $Q_p$ for a fixed $p$ with $0 < p <1/2$.  These graphs $Q_p$ will be disconnected with a single large component whp.
Note that we cannot expect some sort of elegant `symmetry rule' as for Erd\H{o}s-R\'enyi random graphs $G(n,p)$, where (roughly speaking), given the size of the largest component in a supercritical random graph $G(n,p)$, the rest of the graph looks like a subcritical $G(n', p')$ (see for example~\cite[section 5.6]{JLR}): the geometry of the hypercube makes life more interesting and complicated.

We denote the number of vertices in a graph $G$ by $v(G)$, and call this the \emph{size} of $G$; and denote the number of edges by $e(G)$. In $Q_p$, we order the components by size (where components having the same size are ordered say by the position of the `smallest' vertex of each component in some canonical ordering of the vertices). Denote the $j$-th component by $\cL_j$ and let $L_j=v(\cL_j)$ be the size of $\cL_j$ (where $\cL_j =\emptyset$ and $L_j=0$ if $G$ has less than $j$ components). The \emph{giant} component is $\cL_1$.  The \emph{fragment}  $\cZ$ is the graph formed by all the components other than $\cL_1$, and we let $Z=v(\cZ) = 2^d - L_1$. Let $X_t$ denote the number of components of  $Q_p$ of size $t$, and let $\mu_t= \Ex[X_t]$. 
Let $X = \sum_{t \geq 1} X_t$ be the total number of components of $Q_p$.
%$\cZ$ (so $Q_p$ has $X+1$ components).  
Finally let $q=1-p$. 

Observe that $\mu_1=(2q)^d$; and that $\mu_1 \to \infty$ as $d \to \infty$, since $2q>1$.
%{\CC Since $p<1/2$ we have $2q>1$, and so $\mu_1 =(2q)^d \to \infty$ as $d \to \infty$, exponentially quickly.}
The quantity $\M$ defined by
\begin{equation} \label{def.mp}
  \M %=\left\lfloor -1/\log_2 q \right\rfloor
   = \left\lfloor 1/\log_2 (1/q) \right\rfloor
\end{equation}
is central to our results.  Observe that $\M$ is large for small $p$ and decreases to 1 as $p$ increases to 1/2.
For an integer $t$, we have $2q^t \ge 1 \Leftrightarrow t \le \M$. In particular, we always have $\M \geq 1$ since $2q>1$; % $p<\frac12$; and $\M \geq 2$ if and only if $2q^2 \geq 1$, that is $p \leq 1-1/\sqrt{2} \approx 0.29$.
%It is easy to see that for $t =1, 2$ the expected number $\mu_t = \mu_t (d)$ of components of size $t$ in $Q_p$ is given by
%\begin{equation} \label{eqn.mu1mu2}
%\mu_1=(2q)^d \;\; \mbox{ and } \;\; \mu_2 = (p/2q^2) \, d (2q^2)^d.
%\end{equation}
%{\CC Since $p<1/2$ we have $2q>1$, and so $\mu_1 =(2q)^d \to \infty$ as $d \to \infty$, exponentially quickly.}

Weber~\cite{Weber1986} showed that whp the fragment size $Z$ satisfies $Z \sim \mu_1$ (that is, $Z = (1+o(1)) \mu_1$), the second largest component size $L_2$ satisfies $L_2=\M$, and the number $X_t$ of components of size $t$ satisfies $X_t \sim \mu_t = \Theta( d^{t-1} (2q^t)^d)$ for each $t=1,\ldots,\M$; and it follows that the total number $X$ of components satisfies $X \sim \mu_1$ whp.  We much extend and sharpen these results, presenting our results in six theorems.  Weber's results in~\cite{Weber1986} are contained within Theorems~\ref{thm.C1} and~\ref{cor.normal} below.
%\m{\CC need more theorems?}
  (Weber later introduced also a probability for vertices to appear in the random subgraph of $Q^d$~\cite{Weber1992}, but we do not pursue that extension here.) 

Our first three theorems concern the global behaviour of components in $Q_p$; the next two theorems %Theorems 4 and 5, 
concern more local behaviour (and are needed to prove the earlier ones); and our last theorem, Theorem~\ref{thm.jointdistrib}, concerns the joint distribution of random variables like the $X_t$.

 Throughout, we fix $0<p<1/2$ and let $q=1-p$.  The first theorem can be introduced now, with no further definitions.  It describes the total number $X$ of components in $Q_p$, the size $Z=2^d-L_1$ of the fragment, and the size $L_2$ of the second largest component.  Note that, as $d \to \infty$, we have $d \ll \mu_1$ and so $\sqrt{d \mu_1} \ll \mu_1$.

\begin{theorem}\label{thm.C1}
For fixed $0<p<1/2$, the random graph $Q_p=Q_p^d$ satisfies the following.
\begin{enumerate}[(a)]
\item \label{3tp:1} 
Let $Y$ be either the number $X$ of components of $Q_p$ or the fragment size $Z$. Then
$\Ex[Y]= \mu_1(1+ \Theta(dq^d))$; and for each $\eps>0$ we have $|Y- \Ex[Y]| < \eps \sqrt{d\mu_1}$ wvhp. 
\item \label{3tp:2}
The second largest component size $L_2$ in $Q_p$ satisfies $L_2=\M$ wvhp,
where $\M$ is as in~(\ref{def.mp}). Also, the mean and variance satisfy $\, |\Ex[L_2] - \M | = e^{-\Omega(d)}$ and $\Var(L_2)= e^{-\Theta(d)}$.
\end{enumerate}
\end{theorem}
%\m{\CC more on mean in `extras'}

Our second theorem concerns how the fragment sits in $Q^d$.
How much do the components of the fragment cluster together?
How far is it typically from a fixed vertex to the fragment $\cZ$ of $Q_p$?  Given a vertex $u$ in $Q^d$ and $r>0$, the \emph{$r$-ball $B_r(u)$ around $u$} is the set of vertices $v$ at graph distance at most $r$ from $u$ (in $Q^d$).
Recall that, for $0<\eta<1$, the \emph{entropy} $h(\eta)$ is defined to be $-\eta \log_2 \eta - (1-\eta) \log_2 (1-\eta)$, and it is strictly increasing on $(0,\tfrac12)$ with image $(0,1)$. 
Let $\eta^*=\eta^*(p)$ be the unique solution to $h(\eta)=\log_2 \frac1{1-p}$ %$(1/q)$ 
with $0<\eta<\tfrac12$. 
For example, if $p=\frac14$ then $\eta^*  \approx 0.08$.
\begin{theorem}\label{thm.C2}
For fixed $0<p<1/2$, the random graph $Q_p=Q_p^d$ satisfies the following.
\begin{enumerate}[(a)]
\item \label{3tp:2b} There exists $\delta=\delta(p)>0$ such that 
%the following holds. For $r=r(d)=\delta d/\log d$, 
wvhp each $\delta d$-ball in $Q^d$ contains at most $\M$ vertices of the fragment.
\item \label{item.dtof} %\marginpar{\color{red} added $\gamma$}
For each $\eps>0$ there is $\gamma=\gamma(\eps,p)>0$ such that wvhp a proportion at most $e^{-\gamma d}$ of the vertices in $Q^d$ are within distance $(\eta^* - \eps) d$ of the fragment $\cZ$, but all vertices are within distance $(\eta^* + \eps) d$.  (All distances are in $Q^d$.)
\end{enumerate}
\end{theorem}

In part (a) above, clearly wvhp there are $\delta d$-balls containing at least $\M$ vertices of the fragment -- consider for example any ball with centre in a component of size $\M$.  Thus the statement that wvhp no $\delta d$-ball in $Q^d$ contains strictly more than $\M$ vertices of the fragment is saying strongly 
%\marginpar{\color{red}add comment}
that the components of the fragment $\cZ$ do not cluster together in $Q^d$.  For example, wvhp no component of $\cZ$ of size $m_p$ is within distance $\delta n$ of any other component of $\cZ$.

In part (b), many vertices are at a short distance in $Q^d$ from the fragment $\cZ$, including of course the vertices in $\cZ$, but only a very small proportion of the total are at distance at most $(\eta^*-\eps) d$.  However, when $r=(\eta^* + \eps) d$, wvhp every $r$-ball contains a vertex in $\cZ$ (and indeed contains $2^{\Omega(d)}$ vertices in $\cZ$).
%\m{\CC $2^{\Omega(d)}$ isolated vertices}
%\m{\CC prove statement in brackets?}
Overall, the giant gets everywhere, and indeed the fragment is heavily outnumbered everywhere.

The next theorem amplifies part (a) of Theorem~\ref{thm.C1},
concerning the number $X$ of components and the fragment size $Z$. Recall first that, for two random variables $Y$ and $Y'$ taking values in a countable set, the \emph{total variation distance} between their distributions is given by 
\[ d_{TV}(Y,Y')=\frac 12 \sum_k|\Pro(Y=k)-\Pro(Y'=k)|.
\label{3eq:TVD} \]
We use $d_{TV}(Y,\Po(\lambda))$ to denote $d_{TV}(Y,Y')$ where $Y'$ has the Poisson distribution $\Po(\lambda)$ with mean $\lambda$.  Several of our proofs will involve bounding $d_{TV}(Y,\Po(\Ex[Y]))$ for relevant random variables $Y$ (like $X_t$ or $X$), using results on Poisson approximation based on the Stein-Chen method.
By a standard tail bound (see, for example, inequality~(2.9) and Remark 2.6 in~\cite{JLR}),
for any random variable $Y$ and $\lambda>0$, for each $t>0$ we have
\begin{equation} \label{eqn.potail}
\pr(|Y-\lambda| \geq t \sqrt{\lambda}) \leq 2 e^{-t^2/3} + d_{TV}(Y, \Po(\lambda)).
\end{equation}
Also, given a (non-trivial) random variable $Y=Y_d$ we let $Y^*$ denote the natural centred and rescaled version $(Y- \Ex[Y])/\sqrt{\Var(Y)}$.
% $(Y- \Ex[Y])/\Var(Y)^{1/2}$ $(Y- \Ex[Y])/\Var(Y)^{\tfrac12}$. 
It is well known (see for example~\cite{BHJ}) that if $Y_n$ is a sequence of random variables with mean $\lambda_n$ such that $d_{TV}(Y_n,\Po(\lambda_n)) \to 0$ and $\lambda_n \to \infty$ as $n \to \infty$, then $(Y_n-\lambda_n)/\sqrt{\lambda_n}$ is asymptotically standard normal.  Thus if also $\Var(Y_n) \sim \lambda_n$ then $Y^*_n$ is asymptotically standard normal.
\begin{theorem}\label{thm.C3} %\label{3t:1}
Fix $0<p<1/2$ and let $q=1-p$. 
In $Q_p = Q_p^d$, let $Y$ either be the number $X$ of components or be the fragment size $Z$. Then the following properties hold as $d \to \infty$.
\begin{enumerate}[(a)]
\item \label{3tp:1}
$\lambda := \Ex[Y]= (1+ \Theta(dq^d)) \, \mu_1$ and $\, \Var(Y) = (1+O(dq^d)) \, \mu_1$.

\item
$\: d_{TV} (Y, \Po(\lambda))$ is $O(dq^d)$, and $Y^*$ is asymptotically standard normal.
\end{enumerate}
%\m{drop $|Y\!-\!\lambda| < \eps (d \mu_1)^{\frac12}$? keep in proof}
\end{theorem}
\noindent
Observe that part (a) of Theorem~\ref{thm.C1} follows directly from inequality~(\ref{eqn.potail}) and Theorem~\ref{thm.C3}\,: for 
\begin{eqnarray*} 
  \pr(|Y- \Ex[Y]| \geq \eps \sqrt{d\mu_1})
& \leq &
  2 e^{-\frac13 \eps^2 d\, \mu_1/\lambda} + d_{TV}(Y, \Po(\lambda))\\
& \leq &
   e^{-(\frac13 +o(1)) \eps^2 d} + O(dq^d).
\end{eqnarray*}
\smallskip

The remaining theorems concern more local behaviour.
The first counts small components by size.  It is needed in order to prove the earlier theorems. Recall that $X_t$ is the number of components of size $t$ in %the fragment of
 $Q_p$, and $\mu_t= \Ex[X_t]$. 
We noted earlier that $\mu_1=(2q)^d$.  It is not hard to give exact formulae also for $\mu_2$ and $\mu_3$ (assuming $d \geq 2$), namely
\begin{equation} \label{eqn.mu2mu3}
\mu_2 = (p/2q^2) \, d\, (2q^2)^d \;\;\; \mbox{ and } \;\;\; \mu_3 = (p^2/2q^4) \,  d(d\!-\!1)\, (2q^3)^d
\end{equation}
(see also the discussion following Theorem~\ref{thm.C5}).

\begin{theorem} \label{cor.normal}
Fix $0<p<\frac12$, let $q=1-p$, and let $1\le t\le \M$.
% let $X_t$ be the number of components of size $t$ in $Q_p$, and let $\mu_t= \Ex[X_t]$.
Then the following results concerning the number $X_t$ of components of size $t$ in $Q_p=Q_p^d$ hold, as $d \to \infty$.
\begin{enumerate}[(a)]
\item 
$\mu_t =(1+O(\frac1{d})) \, \tfrac{t^{t-2}}{t!} (\tfrac{p}{q^2})^{t-1} \, d^{t-1} (2q^t)^d$ 
and $\Var(X_t) = (1+ O(d^t q^{td})) \mu_t$.

\item For each $\eps >0$, we have
$|X_t-\mu_t| < \eps \sqrt{d \mu_t}$ wvhp, and so also $|X_t-\mu_t| < \eps \mu_t$ wvhp.
\item $d_{TV}(X_t,\Po(\mu_t))=O(d^tq^{td})$, and 
$X_t^*$ is asymptotically standard normal.
\end{enumerate}
\end{theorem}
\noindent
Observe from part (a) that $\mu_t = \Omega(d)$ since $2q^t \geq 1$ (and indeed $\mu_t \gg d$ unless $p=1\!-\!1/\sqrt{2}$ and $t=m_p = 2$), so the first half of part (b) above implies the second half.
For a partial local limit result corresponding to part~(c), see Proposition~\ref{prop.local} at the end of Section~\ref{3sec:smallcomp}.

These results help us to visualise the asymptotic disappearance of small components in $Q_p$ as $p$ increases from 0 to 1/2. For each fixed $p$, there are wvhp a giant component and many small components of every size %possible size and configuration
 up to a  maximum size $\M$.  In particular $\mu_t \to \infty$ as $d \to \infty$ for each $t \leq \M$. We noted that $\M$ is large for small $p$ and decreases to 1 as $p$ increases to 1/2. The typical number of components decreases exponentially as $p$ increases and the maximum size $L_2$ of a component of the fragment drops as $1/\log_2(1/q)$ falls below each integer value. In particular, the last components of size 2 disappear as $p$ increases past $1\!-\!1/\sqrt{2} \approx 0.29$ and the last isolated vertices disappear as $p$ increases past 1/2.
We recall that $Q_{1/2}$ is connected with probability tending to $e^{-1}$ as $d\rightarrow\infty$. Indeed, whp $Q_{1/2}$ consists of $X$ isolated vertices and a connected component of $2^d-X$ vertices, where $X$ has mean value 1 and asymptotic distribution  $\Po(1)$ (see \cite{79}).
\bigskip

\noindent
{\bf Ambient isomorphisms}

We shall in fact prove a much finer and more detailed version of Theorem~\ref{cor.normal}, namely Theorem~\ref{thm.C5}, which uses a natural restricted version of isomorphism for subgraphs of the cube, so that we can consider also how components `sit' in the host hypercube.  We then deduce Theorem~\ref{cor.normal} from Theorem~\ref{thm.C5}. 
 
We call a graph a \emph{cube subgraph} if it is a subgraph of the cube $Q^d$ for some $d$. Let $H$ be a connected cube subgraph. The \emph{support} $S(H)$
%\m{\CC support, use $S(H)$ or $S_H$ or $supp(H)$?}
is the set of indices $i$ such that there is an edge $xy$ in $H$ with $x_i=0$ and $y_i=1$ (that is, $H$ meets both top and bottom faces in the $i$-th coordinate direction). 
Call $|S(H)|$ the \emph{span} of $H$, denoted by $\s(H)$. Note that if $H$ consists of a single vertex then $\s(H)=0$, and otherwise $\s(H) \geq 1$. Indeed, if $v(H)$ is 1, 2 or 3 then $\s(H)=v(H)-1$, whereas for example if $H$ is a 4-vertex path then $\s(H)$ could be 2 or 3.

The \emph{canonical copy} $H^*$ of $H$ is defined as follows.
If $H$ is a single vertex then its canonical copy is the graph $Q^0$ (consisting of a single vertex).
Suppose that $H$ has at least one edge, so $s:=\s(H) \geq 1$.  Let 
%\m{was $\phi=\phi_S$}
$\phi$ be the increasing injection from $[s]$ to $[d]$ with image $S(H)$. Given $x=(x_1,x_2,\ldots,x_d) \in Q^d$ let  $ \phi(x)=(x_{\phi(1)}, x_{\phi(2)},\ldots, x_{\phi(s)}) \in Q^s.$
Then the vertices of the canonical copy $H^*$ are the points $\phi(x)$ where $x$ is a vertex of $H$; and the edges of $H^*$ are the pairs $\phi(x) \phi(y)$ such that $xy$ is an edge of $H$. (Note that the canonical copy is a subgraph of $Q^s$.) See Figure \ref{3fig:Hcopy} for an illustration.

\floatstyle{plain}
\restylefloat{figure}
\begin{figure}[h]
\small
\centering
\begin{tikzpicture}[show background rectangle, scale=1.1, >=stealth]
\fill[black] (0,0) circle(1pt) node[below]{\mbox{\footnotesize$(0,0,0)$}};
\fill[black] (1,0) circle(1pt);
\draw (1.1,0) node[below]{\mbox{\footnotesize$(1,0,0)$}};
\fill[black] (0,1) circle(1pt);
\fill[black] (1,1) circle(1pt);
\fill[black] (0.6,0.2) circle(1pt);
\fill[black] (0.6,1.2) circle(1pt);
\fill[black] (1.6,0.2) circle(1pt) node[right]{\mbox{\footnotesize$(1,1,0)$}};
\fill[black] (1.6,1.2) circle(1pt) node[right]{\mbox{\footnotesize$(1,1,1)$}}; 
\draw [dotted](0,0)--(0,1);
\draw [dotted](0,0)--(1,0);
\draw [dotted](0,1)--(1,1);
\draw [dotted](1,0)--(1,1);
\draw [dotted](0.6,0.2)--(0.6,1.2);
\draw [dotted](0.6,0.2)--(1.6,0.2);
\draw [dotted](0.6,1.2)--(1.6,1.2);
\draw [thick](1.6,0.2)--(1.6,1.2);
\draw [dotted](0,0)--(0.6,0.2);
\draw [dotted](0,1)--(0.6,1.2);
\draw [thick](1,0)--(1.6,0.2);
\draw [dotted](1,1)--(1.6,1.2);
\draw [->] (-2,0)--(-1,0);
\draw [->] (-2,0)--(-2,1);
\draw [->] (-2,0)--(-1.4,0.2);
\draw (-1.5,0) node[below]{\mbox{\footnotesize$1$}};
\draw (-2,0.5) node[left]{\mbox{\footnotesize$3$}};
\draw (-1.4,0.2) node[right]{\mbox{\footnotesize$2$}};
\draw [thick] (4,0)--(5,0)--(5,1);
\draw [dotted] (4,0)--(4,1)--(5,1);
\draw (4,0) node[below]{\mbox{\footnotesize$(0,0)$}};
\draw (5,0) node[below]{\mbox{\footnotesize$(1,0)$}};
\draw (5,1) node[right]{\mbox{\footnotesize$(1,1)$}};
\draw (4.2,0.5) node[right=0cm]{\mbox{\footnotesize$H^*$}};
\draw (1.6,0.6) node[right=0cm]{\mbox{\footnotesize$H$}};
\draw [->] (2.2,0.6)--(3.5,0.6);
\draw (2.9,0.6) node[above=0.02cm]{\mbox{\footnotesize$\phi$}};
\draw (-3,-1) node[right=0cm]{\mbox{\footnotesize$V(H)=\{(1,0,0),(1,1,0),(1,1,1)\}$}};
\draw (6,0.7) node[right=0cm]{\mbox{\footnotesize$S=\{2,3\}$}};
\draw (6,0.2) node[right=0cm]{\mbox{\footnotesize$\phi(1)=2,\,\phi(2)=3$}};
\draw (4,-1) node[right=0cm]{\mbox{\footnotesize$V(H^*)=\{(0,0),(1,0),(1,1)\}$}};
\end{tikzpicture}
\normalsize
\caption{\mdseries A subgraph $H$ of $Q^3$ with canonical copy $H^*$ in $Q^2$ \normalsize}
\label{3fig:Hcopy}
\end{figure}
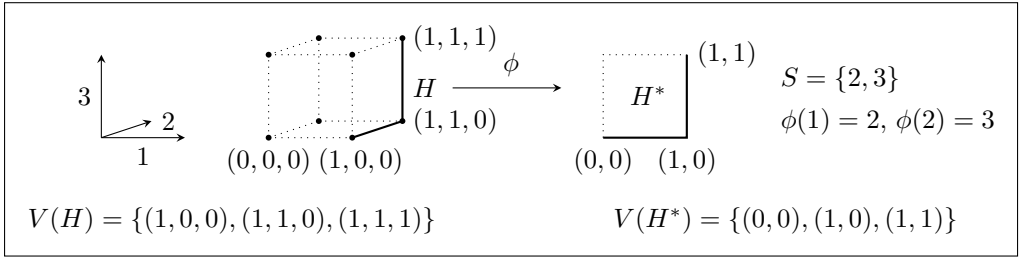
\floatstyle{boxed}
\restylefloat{figure}

We say that connected subgraphs $H_1$ of $Q^{d_1}$ and $H_2$ of $Q^{d_2}$ are \emph{ambient isomorphic} if they have the same canonical copy. Of course, if $H_1$ and $H_2$ are ambient isomorphic then they are isomorphic, but this definition is stronger in that it requires the copies to `sit in the cube' in the same way. For example, let $O$ denote the zero $d$-vector and let $e_k$ denote the $k$th unit $d$-vector: if $i<j$ then the three vertex path $O,e_i,e_i+e_j$ in $Q^d$ has canonical copy the path $(0,0),(1,0),(1,1)$ in $Q^2$ as in Figure \ref{3fig:Hcopy}, and so the original path in $Q^d$ is not ambient isomorphic to the path $O,e_j,e_i+e_j$ which has canonical copy the path $(0,0),(0,1),(1,1)$.
There are four ambient isomorphism classes of three-vertex paths. Observe that if $s=\s(H)$ then there is a unique subgraph of $Q^s$ ambient isomorphic to $H$ (namely the canonical copy of $H$).

Our fifth theorem concerns numbers of components ambient isomorphic to given connected cube subgraphs $H_i$.  Note that any two subcubes of $Q^d$ with the same dimension are ambient isomorphic.  Weber~\cite{Weber1987} considered Poisson convergence of the number of subcube components of $Q^d_p$ of a given dimension, for a range of values of $p$ which could depend on $d$.  Here we keep $p$ fixed, but we consider all kinds of components.
Recall that $(d)_k$ means $d(d-1) \cdots(d-k+1)$.

\begin{theorem}\label{thm.C5} \label{3t:2}
Let  $0<p<1/2$ and $q=1-p$. % and let $\M$ be as in~(\ref{def.mp}).
 Let $r \ge 1$ and let $H_1,H_2,\dots,H_r$ be pairwise non-ambient-isomorphic connected cube subgraphs each with at most $\M$ vertices. Let $t=\min_{i\in[r]}v(H_i)$ and $s=\max\{\s(H_i):v(H_i)=t\}$. (All these quantities are fixed, not depending on d.)
 
 For each $i$, let $Y_i=Y_i(d)$ be the (random) number of components of $Q_p^d$ ambient-isomorphic to $H_i$.  
 Let $Y=Y(d)=\sum_{i}Y_i$ and let $\lambda=\lambda(d)=\Ex[Y]$. Then the following hold.
\begin{enumerate}[(a)]
\item  There is a constant $c>0$, given explicitly in equations~(\ref{eqn.beta}) and~(\ref{eqn.betasum}) below, such that
$\, \lambda = (1+O(1/d)) \, c \, (d)_s (2q^t)^d$;
and if $t$ is 1, 2 or 3 then $s=t-1$, and we may replace the error bound $O(1/d)$ by $O(dq^d)$.
Also $\Var(Y) = (1+ O(d^tq^{td}))\, \lambda$.

\item For each $\eps >0$, we have 
$|Y-\lambda| < \eps \sqrt{d \lambda}$ wvhp, and so also $|Y-\lambda| < \eps \lambda$ wvhp.

\item
$d_{TV}(Y,\Po(\lambda))=O(d^tq^{td})$, and 
$Y^*$ is asymptotically standard normal.

%\item if $t \leq \M$ then $\lambda \to \infty$ and $\tilde{Y}$ is asymptotically standard normal.
\end{enumerate}
\end{theorem}
\noindent
By part (a), $\lambda$ is $\Omega(d)$ (and indeed $\lambda$ is $\Omega(d^2)$ except if $p=1-1/\sqrt{2}$ and $t=m_p =2$), so the first half of part (b) implies the second half (as with Theorem~\ref{cor.normal}).
See Lemma~\ref{lem.new9a} for a fuller version of Theorem~\ref{thm.C5}, which considers more information about the components counted.  That lemma, together with the estimates of $\mu_t$ from Lemma~\ref{3lem:Xi}, will yield Theorem~\ref{cor.normal},
by letting $H_1,\ldots,H_r$ list all the $t$-vertex connected canonical cube subgraphs, so that the random variable $Y$ in Theorem~\ref{thm.C5} is $X_t$.

The constant $c$ in part (a) may be specified as follows.
Let $I^*=\{i \in [r] :  v(H_i)=t,\, \s(H_i)=s \}$.  For each $i \in I^*$, let $e'(H_i)$ be the number of edges of $Q^d$ not in $H_i$ but with both end vertices in $H_i$, and let
\begin{equation} \label{eqn.beta}
\beta_i = \frac1{2^{s} s!} \, \bigg(\frac{p}{q^2}\bigg)^{e(H_i)} \, \bigg(\frac1{q}\bigg)^{e'(H_i)}.
\end{equation}
Now let
\begin{equation} \label{eqn.betasum}
c= \sum_{i \in I^*} \beta_i.
\end{equation}
If $t=1$ then $c=1$.  If $t=2$ then $c= p/2q^2$, so $\lambda \sim (p/2q^2) \, d (2q^2)^d$.
If $t=3$ then $1 \leq |I^*| \leq 4$ and each $\beta_i= \tfrac18 (p/q^2)^2$, so if $|I^*|=4$ we have $\lambda \sim (p^2/2 q^4) \, d^2 (2q^3)^d$. These results are in accord with~(\ref{eqn.mu2mu3}). 

In Theorem~\ref{cor.normal} we saw that wvhp in $Q_p$ there are components of each size up to $\M$.  In Theorem~\ref{3t:2} we see in much more detail that each connected cube subgraph of size at most $\M$, with its way of sitting within the host hypercube, appears wvhp as a component of $Q_p$.

What we call ambient isomorphism could be called `ordered ambient isomorphism', since we insist that the injection $\phi$ in the definition is increasing.  If we drop this requirement then essentially the same results hold (mutatis mutandis), since the new isomorphism classes are unions of the old ones.  When we deduce Theorem~\ref{cor.normal} from Theorem~\ref{3t:2}/Lemma~\ref{lem.new9a}, we may think of this as relaxing \emph{ambient isomorphism} all the way to \emph{isomorphism}.

Given a connected cube subgraph $H$, let $p_H = p_H(d)$ be the probability that $Q_p$ has a component ambient isomorphic to $H$.
When $p$ is fixed with $0<p<\frac12$, by Theorem~\ref{3t:2}, either $p_H$ or $1-p_H$ is $e^{-\Omega(d)}$. 
To see this, let $t=v(H)$, let $Y$ be the number of components ambient isomorphic to $H$ 
%\m{\CC mention isomorphic, not necessarily ambient isomorphic?} 
and $\lambda= \Ex[Y]$. If $t>\M$ then $2q^t<1$, so $\pr(Y \geq 1) \leq \lambda =e^{-\Omega(d)}$; and if $t \leq \M$ then $\lambda \to \infty$ (as we saw above), and by part (b) of Theorem~\ref{thm.C5} wvhp $Y \geq \lambda/2 >0$.  The situation described above is in contrast with the situation at $p=\tfrac12$, when (as we noted earlier) the number of isolated vertices has asymptotic distribution $\Po(1)$.  
\medskip

\noindent
\emph{Joint distribution of components} 

We saw in Theorem~\ref{cor.normal} that, for each $t=1,\ldots,\M$ the number $X_t$ of components of $Q_p$ of size $t$ has close to the Poisson distribution $\Po(\mu_t)$, where $\mu_t = \Ex[X_t]$. 
In fact more is true: the joint distribution of $X_1,\ldots,X_{\M}$ is close to a product of these distributions.
Write ${\cal L}(X_1,\ldots,X_{\M})$ for the joint law of $X_1,\ldots,X_{\M}$; and write $\prod_{j=1}^{\M} \Po(\mu_j)$ for the joint distribution of independent random variables $\Po(\mu_j)$.  We shall see that
\begin{equation}\label{eqn.jointXt}
  d_{TV}\big({\cal L}(X_1,\ldots,X_{\M}),\prod_{j =1}^{\M} \Po(\mu_j)\big) =O(d^2q^d).
\end{equation}
Thus, the numbers of components  in the fragment of each size $t$ are asymptotically independent, with a Poisson distribution for $t \leq \M$, and identically 0 for $t>\M$.
Indeed, we have the following much more detailed theorem concerning the small components, in the spirit of Theorem~\ref{thm.C5}.
% (which immediately implies~(\ref{eqn.jointXt})).
Note that there is a finite set of canonical cube subgraphs with at most $\M$ vertices.
\begin{theorem}\label{thm.jointdistrib}
  Let $H_1,\ldots,H_r$ be a list of $r \geq 1$ distinct canonical cube subgraphs each with at most $\M$ vertices. For each $j\in[r]$, let $Y_j$ be the random number of components of $Q_p = Q_p^d$ ambient isomorphic to $H_j$, with mean $\lambda_j$.  Let $t^*=\min_j v(H_j)$. Then
\begin{equation}\label{eqn.jointYi}
  d_{TV}\big({\cal L}(Y_1,\ldots,Y_{r}),\prod_{j=1}^{r} \Po(\lambda_j)\big) =O(d^{t^*+1} q^{t^* d}).
\end{equation}
\end{theorem}
When the $H_j$ include all the canonical cube subgraphs of size up to $\M$ (so $t^*=1$), Theorem~\ref{thm.jointdistrib} directly implies~(\ref{eqn.jointXt}). We cannot quite use Theorem~\ref{thm.jointdistrib} to deduce our earlier individual bounds on $d_{TV}$, for example on $d_{TV}(X_t, \Po(\mu_t))$ in Theorem~\ref{cor.normal} part (c), since in the bound~(\ref{eqn.jointYi}) there is an `extra' factor $d$.
\bigskip

\noindent
\bfseries{Notation }\mdseries
\smallskip

We use standard notation throughout.  For non-negative functions $f$ and~$g$, we say that $f(d)=\Omega(g(d))$ if $\liminf_{d\to\infty} f(d)/g(d)>0$, and
$f(d)=\Theta(g(d))$ if both $f(d)=\Omega(g(d))$ and $g(d)=\Omega(f(d))$.  Also, we write $f\ll g$ if $f(d)=o(g(d))$. 
\bigskip

\noindent
\bfseries{Plan of the paper }\mdseries
\smallskip

Section~\ref{sec.prelim} gives preliminary results, first concerning subgraphs in the hypercube $Q^d$, and then concerning the variance of counting random variables and their closeness to a Poisson distribution.
In Section \ref{3sec:smallcomp}, 
Lemma~\ref{lem.new9a} gives several results concerning numbers of components ambient-isomorphic to a given list of connected cube subgraphs.
Lemma \ref{3lem:Xi} gives quite precise results on the expected value %and the variance 
of $X_t$ for $1 \le t \le \M$. These lemmas allow us to prove Theorem~\ref{3t:2}, and then Theorem~\ref{cor.normal}, at the end of the section.

In order to prove Theorems~\ref{thm.C1},~\ref{thm.C2} and~\ref{thm.C3}
%the earlier theorems, %Theorem~\ref{3t:1} 
we must show that with tiny failure probability there is just one component of size strictly greater than $m_p$. 
%\m{\CC C 5 May}
To do this, in Section~\ref{sec.nolarge} we call a vertex `good' if its degree in $Q_p$ is at least half the expected value $dp \,$.  We show that, with tiny failure probability, all good vertices are in the same component; and then deduce that, for a suitable constant $N$, with tiny failure probability each component of the fragment has size at most~$N$.  From this result, we see in particular that wvhp $m_p$ is an upper bound for the size $L_2$ of a second largest component. 
In Section~\ref{sec.lastproofs} we complete the proofs of Theorems~\ref{thm.C1},~\ref{thm.C2} and ~\ref{thm.C3}.
In Section~\ref{sec.joint} we consider joint distributions and prove Theorem~\ref{thm.jointdistrib}.
Finally, Section~\ref{sec.conclusion} contains some very brief concluding remarks. 

These investigations arose from work on multicommodity flows in the cube $Q^d$ when edges have independent random capacities, see~\cite{firstpaper}.

\section{Preliminary results}\label{sec.prelim}

\subsection{Preliminary results on the hypercube $Q^d$}
\label{subsec.prelimh}

Let us first consider $\s(H)$ for a {connected} cube subgraph $H$.
We have already noted that $\s(H) = v(H)-1$ if $v(H)$ is 1, 2 or 3.
It is easy to see that always $\s(H) \leq v(H) -1$, and the inequality is strict if $H$ is not a tree (since any cycle contains at least two edges in some dimension).  If we have equality we call $H$ a \emph{spreading tree}.
Note that each edge of a spreading tree sits in a distinct dimension, and if $T_1$ and $T_2$ are ambient isomorphic trees then $T_1$ is spreading if and only if $T_2$ is spreading.

What are the subcubes in $Q^d$?
If we are given $S \subseteq [d]$ and ${z} \in \{0,1\}^{[d] \setminus S}$, then clearly the vertices ${x}$ %$\in \{0,1\}^{[d]}$ 
such that $x_j=z_j$ for each $j \in [d]\setminus S$ form a subcube isomorphic to $Q^{|S|}$. 
We shall need to consider such `cylinder' subcubes, for example in the proof of Lemma~\ref{lem:span}.  As an aside, let us note that each cube subgraph $H$ isomorphic to a hypercube $Q^s$ is obtained in this way.  Since $v(H)=2^s$, this is easily seen to be equivalent to showing that $H$ has span $s$; and it is a straightforward exercise to show the latter.

\begin{proposition} \label{prop.subcubenew}
Let $H$ be a subgraph of $Q^d$ isomorphic to a hypercube $Q^s$.  Then $\s(H)=s$.\qed
%With notation as above, $H$ is the subgraph of $Q^d$ induced by~$W$.
\end{proposition}

Next we investigate the number $n_H=n_H(d)$ of subgraphs of $Q^d$ ambient-isomorphic to a given subgraph $H$, the number of subgraphs which are spreading trees of a given size $t$, and the total number of connected subgraphs of size $t$. 

\begin{lemma}\label{lem:span}
\begin{enumerate}[(a)]
	\item For each connected subgraph $H$ of $Q^d$,
	%{\CC the number $n_H$ of subgraphs of $Q^d$ ambient isomorphic to $H$ satisfies} 
$n_H=2^{d-s}\binom{d}{s}$, where $\s(H)=s$.
\item For each $d \geq t-1 \geq 0$, the number of ambient-isomorphism classes of spreading trees of size $t$ in $Q^d$ is $2^{t-1}t^{t-3}$.
\item For each $d \geq t-1 \geq 0$, the number of subgraphs of $Q^d$ which are spreading trees of size $t$ is
		$\, 2^d \, t^{t-3} \binom{d}{t-1}$.
\item For each fixed $t \geq 1$, the number of connected subgraphs of $Q^d$ of size $t$ is $\, 2^d \, t^{t-3}\binom{d}{t-1}(1+O(d^{-1}))$.
\end{enumerate}
\end{lemma}
We see from parts (c) and (d) above that the population of connected subgraphs of a given size $t$ in $Q^d$ is asymptotically dominated by spreading trees.
\begin{proof}
We first recall that any cube subgraph of size $t$ can be embedded in $Q^{t-1}$ and so, for $d\ge t-1$, 
 the number of pairwise non-ambient-isomorphic connected cube subgraphs of size $t$ depends only on $t$.
\smallskip

\noindent (a)	
There is a single ambient-isomorphic copy of $H$ in each (cylinder) subcube $Q^{s}$ of $Q^d$, and there are  $2^{d-s}\binom{d}{s}$ copies of $Q^{s}$ in $Q^d$, so $n_H=2^{d-s}\binom{d}{s}$, as required.
\smallskip

\noindent (b) 
By Cayley's formula there are $t^{t-2}$ trees on the set $\{0,1,2,\dots,t-1\}$ of $t$ vertices.	% which we label $\{0,1,2,\dots,t-1\}$. 
Given one of these trees, call vertex $0$ the root and move the other vertex labels onto the edge leading towards the root. This constructs a vertex-rooted, edge-labeled tree, with edge-labels $1,2,\dots,t-1$. The construction is reversible, so there are exactly $t^{t-2}$ such trees.
	
Given such a rooted, edge-labeled tree $T$, we choose a vertex in $Q^{t-1}$ for the root, then use the labels of the edges to specify the `dimension' in which that edge exists. This defines a $t$-vertex rooted spreading tree, and all the rooted trees constructed are distinct; and furthermore every $t$-vertex rooted spreading tree in $Q^{t-1}$ can be constructed in this way. Thus there are $2^{t-1}t^{t-2}$ $t$-vertex rooted spreading trees in $Q^{t-1}$, and so $2^{t-1}t^{t-3}$ $t$-vertex unrooted spreading trees; and of these unrooted trees, no two distinct ones are ambient-isomorphic since they have span $t-1$ and so are their own canonical copies.
\smallskip

\noindent (c)	 
By parts (a) and (b), the number of $t$-vertex spreading trees in $Q^d$ is
\[ 2^{t-1} t^{t-3} \cdot 2^{d-(t-1)}\binom{d}{t-1} = 2^d t^{t-3} \binom{d}{t-1}.\]
\smallskip

\noindent (d)
If $T$ is a spreading tree of size $t$, and $H$ is a connected cube subgraph of size $t$ with $\s(H)<t-1= \s(T)$, then $n_H/n_T=O(d^{-1})$ by part (a). The number of ambient-isomorphism classes of connected subgraphs of $Q^d$ of size $t$ does not depend on $d$ for $d \geq t-1$; and thus the contribution to the total number of connected subgraphs of $Q^d$ of size $t$ by those with span less than $t-1$ is $O(d^{-1})$ of the total.
\end{proof}

We will need one more lemma which we will apply to the hypercube $Q^d$.  This result is `folk knowledge' (and indeed a more precise result is known,
see equation~(\ref{eqn.fdt})) but we give a short combinatorial proof here for completeness.
\begin{lemma}\label{lem.fewtrees}
Let the graph $G$ be rooted at vertex $r$ and have maximum degree at most $d$.  Then for each non-negative integer $t$, the number of subtrees containing~$r$ and exactly $t$ other vertices is at most $(ed)^t$.
\end{lemma}
\begin{proof}
We first show (a) that  the number of $(t+1)$-vertex subtrees in $G$ containing $r$ is at most the number $f(d,t+1)$
%\m{\CC new $f(d,t+1)$}
 of $(t+1)$-vertex subtrees containing the root in an infinite $d$-ary tree $T^{\infty}$; and then show (b) that 
%the latter number 
$f(d,t+1)$ is at most the number of points $x \in \{0,1\}^{td}$ with $t$ 1's.  The number of such points is $ \binom{td}{t} \leq (ed)^t$. Clearly we may assume that $t \geq 1$.

The \emph{path tree} $T(G,r)$ \cite{godsil81} has a vertex
%\m{\CC replaced node by vertex (twice)}
for each path $P$ in $G$ from $r$, adjacent to each vertex corresponding to a path extending $P$ by one edge; and as the root has the vertex corresponding to the path with a single vertex $r$.  It is easy to see that, for each tree in $G$ containing $r$, there is a corresponding tree in $T(G,r)$ containing the root.  Thus the the number of $(t+1)$-vertex subtrees in $G$ containing $r$ is at most the number of $(t+1)$-vertex subtrees containing the root in $T(G,r)$; and since $T(G,r)$ embeds in $T^{\infty}$, part~(a) of the proof follows.

For part (b), let $T$ be a $(t\!+\!1)$-vertex subtree in $T^{\infty}$ containing the root.  We may suppose that $T^\infty$ is embedded in the plane, with the root at the top and children listed in order from left to right. We construct $x(T) \in \{0,1\}^{td}$ with $t$ 1's as follows.  Initially the vector $x$ is null and the list $L$ contains just the root.
We repeat the following $t$ times.  Remove the first vertex $v$ in $L$, and let $y \in \{0,1\}^d$ indicate its children (with a 1 for each child): append $y$ to $x$ and append the children to $L$ (listed in order). The output $x(T)$ is the final value of $x$. % which is in $\{0,1\}^{td}$ with $t$ 1's.
Clearly we can reconstruct $T$ from $x(T)$, so the number of possible trees $T$ is at most the number of possible vectors $x(T)$, which completes the proof.
\end{proof}
\noindent
We shall not use this result here, but the precise value of $f(d,t)$ is given by
\begin{equation} \label{eqn.fdt}
 f(d,t) = \frac1{(d-1)t+1} \binom{dt}{t} \;\;\; \mbox{ for each } d, t \geq 1 \,,
\end{equation}
see exercise 11 in~\cite[section 2.3.4.4]{Knuth} (pages 397 and 589).

\subsection{Preliminary results on variance and approximation to Poisson distribution}
\label{subsec.prelimvar}

Let $(A_i: i \in I)$ be a family of events with a dependency graph $L$ (so that $A_i$ and $A_j$ are independent if $i$ and $j$ are not adjacent in $L$ and $i \neq j$).  Write $i \sim j$ if $i$ and $j$ are adjacent in $L$. For each~$i$, let $\pi_i= \pr(A_i)$ and let $\ind_i$ be the indicator function of $A_i$.  Let $X=\sum_i \mathbb{I}_i$ (in this subsection we do not use $X$ as the number of components in $Q_p$). Then
\begin{align}
\Var(X) &=\sum_i\sum_j\left( \pr(A_i\land A_j) - \pi_i \pi_j \right) \nonumber\\
=&\sum_i (\pi_i - \pi_i^2)+\sum_i\sum_{j \sim i }(\Pro(A_i \land A_j) - \pi_i \pi_j )\nonumber\\
=& \, \Ex[X]+ \Delta^{+} - \Delta^-,
\label{eqn.var}
\end{align}
where
\begin{equation} \label{eqn.deltaplus}
\Delta^+ = \sum_i\sum_{j \sim i} \Pro(A_i \land A_j)
\end{equation}
and  
\begin{equation} \label{eqn.deltaminus}
%\Delta^- = \sum_i \Pro(A_i)^2 + \sum_i\sum_{j \sim i}\Pro(A_i) \Pro (A_j).
\Delta^- = \sum_i \pi_i^2 + \sum_i\sum_{j \sim i}\pi_i \pi_j.
\end{equation}

The following lemma is essentially Theorem 6.23 of~\cite{JLR}, proved by the Stein-Chen method, which shows that a sum $X$ as above has close to a Poisson distribution, provided $\Delta^+$ and $\Delta^-$ are small.   
\begin{lemma} 
\label{lem.TV}  
With notation as above, and letting $\lambda=\Ex[X]$,
we have
\begin{eqnarray*}
d_{TV}(X,\Po(\lambda)) & \leq &
\min\{\lambda^{-1}\!,1\} \left( \Delta^+ + \Delta^- \right).
\end{eqnarray*}
\end{lemma}

We shall also need a minor extension of the above.  Suppose that we are given a family $(t_i:i \in I)$ of positive integers, and let $\tilde{X}=\sum_i t_i  \mathbb{I}_i$. Then much as above, we have
\begin{align}
\Var(\tilde{X}) &= \sum_i\sum_j  t_it_j \left( \pr(A_i \land A_j) - \pi_i \pi_j \right) \nonumber\\
=& \sum_i t_i^2 (\pi_i - \pi_i^2) + \sum_i\sum_{j \sim i }  t_it_j  (\pr(A_i \land A_j) - \pi_i \pi_j )\nonumber\\
=& \, \Ex[\tilde{X}] + \tilde{\Delta}^+ - \tilde{\Delta}^- \label{eqn.varti}
\end{align}
where
\begin{equation} \label{eqn.Deltaplus}
 \tilde{\Delta}^+ =   \sum_i t_i(t_i-1) \pi_i + \sum_i\sum_{ j \sim i} t_it_j \, \pr(A_i \land A_j)
\end{equation}
and
\begin{equation} \label{eqn.Deltaminus}
\tilde{\Delta}^- = \sum_i t_i^2 \pi_i^2+\sum_i\sum_{j \sim i} t_it_j \, \pi_i \pi_j.
\end{equation}

\begin{lemma}\label{lem:TV2}
With notation as above, and letting $\lambda=\Ex[\tilde{X}]$, we have
\begin{eqnarray*}
d_{TV}(\tilde{X},\Po(\lambda))
& \leq & 
\min\{\lambda^{-1}\!,1\} \big( \tilde{\Delta}^+ + \tilde{\Delta}^- \big).
\end{eqnarray*}
\end{lemma}
\begin{proof}
  Replace each event $A_i$ by $t_i$ identical (not independent) copies.  Note that, for each $i$, the $t_i$ copies of $A_i$ are dependent, and so they are adjacent to each other in the natural extended dependency graph.  Now apply Lemma~\ref{lem.TV}.
\end{proof}

%%%%%%%%%%%%%%%%%%%%%%%%%%%%%%%%
%%%%%%%%%%%%%%%%%%%%%%%%%%%%%%%%

\section{The numbers of small components}\label{3sec:smallcomp}

The first lemma in this section, Lemma~\ref{lem.new9a}, gives expected values and variances for the numbers of small components in certain ambient-isomorphism classes, and for the number of vertices in such components; and gives some results on approximation by a Poisson distribution.
The second lemma uses Lemma~\ref{lem.new9a}, together with counting results from Subsection~\ref{subsec.prelimh}, to deduce results corresponding to those in Lemma~\ref{lem.new9a} when we consider \emph{all} components of a given size.
Using these lemmas we prove Theorem~\ref{thm.C5}  and then Theorem~\ref{cor.normal}. 

In Lemma~\ref{lem.new9a}, we consider both the numbers of components in $Q_p$ ambient isomorphic to given graphs, and the total numbers of vertices in such components. 
%Recall that $(d)_s$ means $d(d-1) \cdots(d-s+1)$.

\begin{lemma}\label{lem.new9a}
Let $0<p<\frac 12$ and let $q=1-p$. Let $r$ be a positive integer and let $H_1,H_2,\dots,H_r$ be pairwise non-ambient-isomorphic connected cube subgraphs.
For each $i \in [r]$, let $s_i=\s(H_i)$, and recall that $e'(H_i)$ is the number of cube edges not in $H_i$ but with both end vertices in $H_i$. (All these quantities are fixed, not depending on d.)

For each $i \in [r]$, let $Y_i$ be the number of components of $Q_p$ ambient-isomorphic to $H_i$.
Let $t=\min_i v(H_i)$, and let $s= \max \{s_i: v(H_i)=t\}$.  Let $I^*=\{i \in [r] :v(H_i)=t, s_i=s\}$, and let
\[ c= \frac1{2^s s!} \sum_{i \in I^*} (p/q^2)^{e(H_i)} q^{-e'(H_i)}.\]
 Then the following hold.
\begin{enumerate}[(a)]
\item  For each $i \in [r]$, once $d \geq s_i$ we have
\[ \Ex[Y_i]= (p/q^2)^{e(H_i)} q^{-e'(H_i)} \, 2^{d-s_i}\binom{d}{s_i}q^{v(H_i) d}.\] 

\item
The sum $\, Y=\sum_{i=1}^r Y_i \,$ satisfies \,
(i) \, $\Ex[Y] = (1+ O(1/d))\, c\, (d)_s (2q^t)^d$, 
(ii) $\Var(Y)=  (1+O(d^{t} q^{t d})) \, \Ex[Y]$, and (iii) $d_{TV}(Y, \Po(\Ex[Y])) = O(d^{t} q^{t d})$.
Furthermore, if $t$ is 1, 2 or 3 then $s=t-1$ and in the expression for $\Ex[Y]$ we can improve the error term, 
so $\Ex[Y] = (1+ O(dq^d))\, c\, (d)_{t-1} (2q^t)^d$.

\item
The weighted sum $\tilde{Y} = \sum_{i=1}^r v(H_i) Y_i$ satisfies 
(i) $\Ex[\tilde{Y}]= (1+O(dq^d)) \, t \, \Ex[Y]$.
Furthermore, if $t=1$ then (ii) $\Var(\tilde{Y})= (1+O(dq^d)) \, \Ex[\tilde{Y}]$ and  (iii) $d_{TV}(\tilde{Y}, \Po(\Ex[\tilde{Y}])) = O(dq^d)$.

\end{enumerate} 
\end{lemma}

\begin{proof}
\noindent
(a) 
Consider a fixed graph $H_i$.  Let $G$ be a subgraph of $Q^d$ which is ambient-isomorphic to $H_i$, and let $A$ be the event that the subgraph of $Q_p$ induced by the vertices of $G$ is exactly $G$, and it is also a component of $Q_p$. Then 
\begin{equation} \label{eqn.prAk}
 \Pro(A)= p^{e(H_i)} q^{e'(H_i)} q^{v(H_i) d-2e(H_i)-2e'(H_i)} = (p/q^2)^{e(H_i)} q^{-e'(H_i)}  q^{v(H_i) d}.
\end{equation}
Hence, by Lemma~\ref{lem:span} part (a)
\[ \Ex[Y_i]= 2^{d-s_i}\binom{d}{s_i} \,\, (p/q^2)^{e(H_i)} q^{-e'(H_i)} \,  q^{v(H_i) d},\]
completing the proof of part (a).
\medskip

\noindent (b)
Observe from part (a) that $\Ex[Y_i] = \Theta\big(d^{s_i} (2q^{v(H_i)})^d \big)$.  Thus the dominant contribution to $\Ex[Y]$ is from graphs $H_i$ with $i \in I^*$ (for if $i \in I^*$ and $j \in I \backslash I^*$, then $\Ex[Y_j]= O(1/d) \, \Ex[Y_i]$).  Using part (a) we now see that
\[ \Ex[Y] = (1+ O(1/d))\, c (d)_s (2q^t)^d = (1+ O(1/d))\, c d^s (2q^t)^d.\]
Now suppose that $t$ is 1, 2 or 3.  If $i \in I^*$ and $j \in I \backslash I^*$, then $v(H_j) >t$ so $\Ex[Y_j]= O(dq^d) \, \Ex[Y_i]$ (note that if $v(H_j)=t+1$ then $s_j \leq s+1$).
Hence
$\Ex[Y] = (1+ O(dq^d))\, c (d)_s (2q^t)^d$.
(If $t \geq 4$ then there could be $t$-vertex graphs $H_i$ with different spans, and if one has span $s-1$ then $\Ex[Y] = (1+ \Theta(1/d))\, c (d)_s (2q^t)^d$.)
\medskip

Now we prove parts (b)(ii) and (b)(iii).
Given $d$, let $\cS=\cS(d)$ be the set of subgraphs of $Q^d$ ambient isomorphic to one of the graphs $H_1,\ldots,H_r$. List the members of $\cS$
as $G_1,\ldots,G_N$ (where $N=N(d)$); and let $A_i$ be the event that $G_i$ is a component of $Q_p$.  For distinct $i,j \in [N]$ let $i \sim j$ if either the vertex sets $V(G_i)$ and $V(G_j)$ intersect or there is an edge of $Q^d$ between them.  Observe that if $i \neq j$ and $i \not\sim j$ then the events $A_i$ and $A_j$ are independent, so we have a dependency graph. Now by~(\ref{eqn.var})
$\Var(Y)= \Ex[Y] + \Delta^{+} - \Delta^-$,
where  $\Delta^{+}$ and $\Delta^-$ are defined in~(\ref{eqn.deltaplus})  and~(\ref{eqn.deltaminus}) respectively.  We next bound $\Delta^+$ then $\Delta^-$.

If $i \ne j$ and the vertex sets $V(G_i)$ and $V(G_j)$ intersect, then $\Pro(A_i \land A_j)=0$, so in the sum for $\Delta^+$ in~(\ref{eqn.deltaplus}) we need consider only the case where the two vertex sets $V(G_i)$ and $V(G_j)$ are disjoint but have connecting edges in $Q^d$ (of which there can be at most $v(G_i) v(G_j)$).  By~(\ref{eqn.prAk}), there is a constant $\alpha$ such that
\begin{equation} \label{eqn.pmax}
\pr(A_i) \leq \alpha \, q^{v(G_i)d} \;\; \mbox{ for each } i.
\end{equation}
Thus, if $i \ne j$ then 
\begin{equation} \label{eqn.pmax2}
 \pr(A_i \land A_j) \leq \pr(A_i) \, \pr(A_j) \, q^{-v(G_i) v(G_j)}
\leq  \pr(A_i)\, \alpha q^{v(G_j) d} q^{-v(G_i) v(G_j)}.
\end{equation}

For each integer $k$ let $h(k)$ be the number of graphs $H_i$ in the list with $v(H_i)=k$.
Observe that for each set $W$ of $k$ vertices of $Q^d$, there are at most $h(k)$ graphs $G_j$ with vertex set $W$, and there are no such graphs $G_j$ if the induced subgraph $Q^d[W]$ of $Q^d$ on $W$ is not connected.
For a given graph $G_i$ of size $t_1$, the number of vertices $v$ in $Q^d$ adjacent  to vertices in $G_i$ is at most $t_1 d$.  By Lemma~\ref{lem.fewtrees}
%\m{was Lemma~\ref{3lem:numcomp}} 
each vertex $v$ is in at most $(ed)^{t_2-1}$ sets $W$ of $t_2$ vertices such that the induced subgraph $Q^d[W]$ is connected.
% possible %\m{\CC did we just miss this, or .. ?} %vertex sets of connected subgraphs of $Q^d$ of size $t_2$. 
But each such vertex set $W$ is the vertex set of at most $h(t_2)$ graphs $G_j$. Thus each vertex $v$ could be in at most $(ed)^{t_2-1} h(t_2)$ graphs $G_j$ of size $t_2$.
  In the sums below, $t_1$ and $t_2$ run over the possible sizes of the graphs $G_i$ and $G_j$.
From the definition~(\ref{eqn.deltaplus}), and using~(\ref{eqn.pmax2}) and the last observation, we have
\begin{eqnarray*}
\Delta^+ & = &
\sum_{t_1} \sum_{t_2} \sum_{i: v(G_i)=t_1} \sum_{j: j \sim i, v(G_j)=t_2} \pr(A_i \land A_j)\\
%& \leq &\sum_{t_1} \sum_{t_2} \sum_{i: v(G_i)=t_1} \pr(A_i) (t_1 d) (ed)^{t_2 -1} \max\{\pr(A_j): v(G_j)=t_2\} q^{-t_1t_2}\\
& \leq &
\sum_{t_1} \sum_{t_2} \sum_{i: v(G_i)=t_1} \pr(A_i) (t_1 d) (ed)^{t_2 -1} h(t_2) \, \alpha q^{t_2 d} q^{-t_1t_2}\\
& \leq &
(1+o(1)) \sum_{t_1} \sum_{i: v(G_i)=t_1} \pr(A_i) (t_1 d) (ed)^{t-1} h(t) \, \alpha q^{td} q^{-t_1 t}\\
& = &
\Ex[Y] \, O(d^t q^{td}), 
\end{eqnarray*}
that is
\begin{equation} \label{eqn.D+}
\Delta^+ = \Ex[Y] \, O(d^t q^{td}).
\end{equation}

Now consider $\Delta^-$.  By~(\ref{eqn.pmax})
\[ \sum_i \pr(A_i)^2 \leq \sum_i \pr(A_i) \cdot \alpha q^{td} = \Ex[Y] \cdot \alpha q^{td},\]
and, as for $\Delta^+$ except without the factor $q^{-t_1t_2}$
(also including pairs $i,j$ with $V(G_i) \cap V(G_j) \neq \emptyset$), we have
\[ \sum_{i} \sum_{j \sim i} \pr(A_i) \pr(A_j) 
= \Ex[Y] \, O(d^t q^{td});\]
thus
\begin{equation} \label{eqn.D-}
\Delta^- = \Ex[Y] \, O(d^t q^{td}).
\end{equation}
Now that we have~(\ref{eqn.D+}) and (\ref{eqn.D-}), 
from~(\ref{eqn.var}) we have $\Var(Y)= \Ex[Y](1+O(q^{td}d^t))$, and by Lemma~\ref{lem.TV}
we have $d_{TV}(Y, \Po(\Ex[Y])) = O(d^{t} q^{t d})$, as required.
\medskip

\noindent
(c)
The contribution to $\Ex[Y]$ from graphs $H_i$ with $v(H_i)>t$ is $O(dq^d) \cdot \Ex[Y]$, and similarly for $\Ex[\tilde{Y}]$.  This gives equation (c)(i).

For parts (c) (ii) and (iii), we may argue as for parts (b) (ii) and (iii), but using Lemma~\ref{lem:TV2} instead of Lemma~\ref{lem.TV}. Assume that $t=1$. Let $G_i$ and $A_i$ be as before, and let $t_i=v(G_i)$.  Then $\tilde{Y}= \sum_{i=1}^r t_i \ind_{A_i}$. Since the $t_i$ are uniformly bounded, the quantity $\tilde{\Delta}^-$ (as in~(\ref{eqn.Deltaminus})) is at most a constant times the unweighted version $\Delta^-$, 
and similarly for the second term in $\tilde{\Delta}^+$ (as in~(\ref{eqn.Deltaplus})).  For the first term in $\tilde{\Delta}^+$, there is no contribution from the isolated vertices (graphs $G_i$ with $t_i = 1$),
 so the term is $O(d (2q^2)^{d})$\,: but $\Ex[Y] \geq \mu_1 = (2q)^d$, so the term is $O(\Ex[Y] \, d q^{d})$.  Hence by~(\ref{eqn.D+}) and~(\ref{eqn.D-}), both $\tilde{\Delta}^+$ and $\tilde{\Delta}^-$ 
are $O(\Ex[Y] \, d q^{d})$.  
Equation~(\ref{eqn.varti}) %for $\Var(Z)$
 and Lemma~\ref{lem:TV2} now complete the proof.
\end{proof}

Recall that $X_t$ denotes the number of components of size $t$ in $Q_p$, and that $\mu_t=\Ex[X_t]$.   We noted earlier (more than once) that $\mu_1=(2q)^d$, and the precise values of $\mu_2$ and $\mu_3$ are given in~(\ref{eqn.mu2mu3}). 

\begin{lemma}\label{3lem:Xi}
Let $0<p<\frac 12$ and let $q=1-p$. Let $t \geq 1$ be fixed.  Then
\[ \mu_t = (1+O(\tfrac1{d}))  \, \tfrac{t^{t-2}}{t!} \, (\tfrac{p}{q^2})^{t-1} d^{t-1} (2q^t)^d = \Theta(d^{t-1} (2q^t)^d).\]
\end{lemma}
\begin{proof}
%\noindent (a) 
If $H_j$ is a spreading tree of size $t$, then $\s(H_j) =t-1$ and $e'(H_j)=0$, and so
by Lemma~\ref{lem.new9a} (a),
\begin{equation}
\Ex[Y_j] = (p/q^2)^{t-1} 2^{d-t+1}\binom{d}{t\minus 1}q^{td},
\label{3eq:spreading}
\end{equation}
 where $Y_j$ is the number of components of $Q_p$ ambient-isomorphic to $H_j$.
To calculate $\mu_t$ we need to sum $\Ex[Y_j]$ over all the ambient-isomorphism classes of $t$-vertex connected cube subgraphs $H_j$. We see from Lemma~\ref{lem.new9a} (a) (and equation~(\ref{3eq:spreading})) that if $H_j$ is a spreading tree and $H_{j'}$ is not (so $\s(H_{j'}) \leq t-2$) %$s_j< t-1$ 
then $\Ex[Y_{j'}] = O(d^{-1})\, \Ex[Y_j]$.  Thus the only significant terms are those corresponding to ambient-isomorphism classes of spreading trees, 
and by Lemma~\ref{lem:span} (b) there are $2^{t-1}t^{t-3}$ such classes. Hence 
\begin{eqnarray*} 
 \mu_t & = &
 (1+O(\tfrac{1}{d})) \, 2^{t-1}t^{t-3} \, 2^{d-t+1}\binom{d}{t\minus 1}q^{td}(p/q^2)^{t-1}\\
 & = & (1+O(\tfrac{1}{d})) \, d^{t-1} (2q^t)^d (t^{t-3}/(t\!-\!1)!)  (p/q^2)^{t-1},
\end{eqnarray*}
as required.
\end{proof}

Let us now complete the proof of Theorem~\ref{3t:2} and then of Theorem~\ref{cor.normal}.
\smallskip

\begin{proof}[Proof of Theorem~\ref{3t:2}]
In part (a), the expected value is from Lemma~\ref{lem.new9a} part~(b)(i), 
and the variance is from Lemma~\ref{lem.new9a} part (b)(ii); and the first half of part (c) (on Poisson approximation) is from Lemma~\ref{lem.new9a} part (b)(iii).

Consider part (b).
By a Chernoff bound (see for example~inequality~(2.9) and Remark~2.6 of~\cite{JLR}),
\begin{eqnarray*} 
\pr(| Y\!-\!\lambda| \geq \eps  (d \lambda)^{\frac12})
& \leq &
\pr(|\Po(\lambda) - \lambda| \geq \eps (d \lambda)^{\frac12}) + d_{TV}(Y, \Po(\lambda))\\
& \leq &
2 e^{- \eps^2 d/3} + O(d^t q^{td}), %= e^{-\Omega(d)},
\end{eqnarray*}
by the Poisson approximation bound.
Thus
$\, \pr(| Y\!-\!\lambda| \geq \eps  (d \lambda)^{\frac12}) =  e^{-\Omega(d)}$, as required.

Finally, consider the second half of part (c).  Since as $d \to \infty$ we have $\lambda \to \infty$, $d_{TV}(Y,\Po(\lambda)) \to 0$ and $\Var(Y) \sim \lambda$, it follows that $Y^*$ %$(Y-\lambda)/\sqrt{\lambda}$ 
is asymptotically standard normal -- see the discussion before Theorem~\ref{thm.C3}.
This concludes the proof of Theorem \ref{3t:2}.
\end{proof}

\begin{proof}[Proof of Theorem~\ref{cor.normal}]
The expression for the mean $\mu_t$ in part (a) is from Lemma~\ref{3lem:Xi}.  The rest follows directly from Theorem~\ref{thm.C5}, with $H_1,\ldots,H_r$ listing a representative of each ambient-isomorphism class of $t$-vertex connected cube subgraphs.
\end{proof}
\begin{remark} \label{remark}
In Theorem~\ref{3t:2} it was natural to restrict our attention to connected graphs $H_i$ with at most $m_p$ vertices, and similarly in Theorem~\ref{cor.normal} it was natural to restrict our attention to components with at most $m_p$ vertices.  However, both these theorems are based on Lemma~\ref{lem.new9a} in which there are no such restrictions.  Thus in fact both these theorems hold without any such restrictions on the numbers of vertices, apart from in the two places in each theorem where we need the expected value $\lambda$ to be large, namely the second half of part (b) and the second half of part (c) (in each of Theorems~\ref{cor.normal} and~\ref{3t:2}).  We shall use this remark in the proof of Theorem~\ref{thm.C1}.
\end{remark}

We have now proved Theorem~\ref{cor.normal}, which says in particular that the distribution of the number $X_t$  of components in $Q_p$ of size $t$ is close to the Poisson distribution $\Po(\mu_t)$.  From what we have already proved, we can quickly give a first corresponding local limit result, showing that for suitable $t$ we have $\Pro(X_t=\nu) \sim \Pro(\Po(\mu_t)=\nu)$ uniformly over the `central range' of integers $\nu$.
Recall from Theorem~\ref{cor.normal} that $\mu_t = \Theta(d^{t-1} (2q^t)^d)$.
%. 
\begin{proposition} \label{prop.local} 
Let $0<p<1/2$ and let $t$ be an integer with $\M/3 < t \le \M$. Then for any fixed $c>0$
\[ \sup_{\nu} \big|\Pro(X_t=\nu)/\Pro(\Po(\mu_t)=\nu) \, -\!1 \big|  = e^{-\Omega(d)}\]
where the $\sup$ is over integers $\nu$ with $|\nu-\mu_t| \leq c\sqrt{\mu_t}$.
\end{proposition}
\begin{proof} 
Note first that
$\Pro(\Po(\mu_t)=\nu)=\Theta(\mu_t^{-\frac 12})$, uniformly over integers $\nu$ with $|\nu-\mu_t| \leq c\sqrt{\mu_t}$.  By Theorem~\ref{cor.normal} part (c), $d_{TV}(X_t,\Po(\mu_t))=O(d^t q^{td})$, so $|\Pro(X_t=\nu)-\Pro(\Po(\mu_t)=\nu)|=O(d^t q^{td})$ uniformly over integers $\nu$; and hence 
\[ |\Pro(X_t=\nu)/\Pro(\Po(\mu_t)=\nu) \,\, -1|=O(d^t q^{td}\mu_t^{1/2}),\]
uniformly over integers $\nu$ with $|\nu-\mu_t| \leq c\sqrt{\mu_t}$. But $d^t q^{td}\mu_t^{1/2}=O(d^{3t/2}(2q^{3t})^{d/2})=o(1)$ provided $2q^{3t}<1$. Finally, we have $2q^{3t}<1$ if $t > \M/3$ (and indeed if $t=\M/3$ unless $(2q)^{\M}=1$).
\end{proof}

%%%%%%%%%%%%%%%%%%%%%%%%%%%%%%%%%
%%%%%%%%%%%%%%%%%%%%%%%%%%%%%%%%%

\section{The fragment $\cZ$ has no large components}
\label{sec.nolarge}

It will be straightforward to handle components of any fixed size $t>{\M}$. We need to show also that wvhp there are no components in $\cZ$ larger than some constant size (see Lemma~\ref{lem.L2leqN} below).  We use two preliminary lemmas.  Given a spanning subgraph $Q'$ of $Q$, call a vertex $Q'$-\emph{good} if its degree in $Q'$ is at least $dp/2$ and \emph{bad} otherwise. 

\begin{lemma}
The probability that there is a pair of $Q_p$-good vertices at distance at most 3 in $Q$ which are not joined by a path of length at most 7 in $Q_p$ is  $2^{-\Omega(d^2)}$.
\end{lemma}
\begin{proof}
For a vertex $v$ we let $\Gamma(v)$ denote its neighbourhood in $Q_p$. 
Fix vertices $u\neq v$ in $Q$ at distance at most 3.
Consider the case when $d_Q(u,v)=3$ (the other cases are similar).   For convenience, we consider $Q^d$ as a graph on the power set of $[d]$.
We may then suppose wlog that $u=\emptyset$ and $v=\{1,2,3\}$.  Let $A$ and $B$ be sets of at least $dp/2$ neighbours in $Q$ of $u$ and $v$ respectively.

For each $i \neq j$ in $\{4,\ldots,d\}$ with $\{i\} \in A$ and $v \cup \{j\} \in B$, there is a path
\[ \{i\}, \{i,j\}, \{i,j,1\}, \{i,j,1,2\}, \{i,j,1,2,3\}, \{j,1,2,3\}\]
in $Q$, not using any edges incident with $u$ or $v$.  These form at least 
$(|A|-3)(|B|-4) \geq (pd/2 -3)(pd/2 -4)$ 
paths in $Q$ of length 5 between $A$ and $B$; and the paths are pairwise edge-disjoint since each edge identifies the pair $(i,j)$.  But the number of paths is at least $p^2 d^2/5$ for $d$ sufficiently large, and then
\begin{align*}
 \pr(&\mbox{no $u\!-\!v$ path of length 7 in } Q_p \mid \Gamma(u)=A, \Gamma(v)=B)\\
 &\leq 
 (1-p^5)^{p^2d^2/5}  \;\; \leq \;\; e^{-p^7 d^2/5}.
\end{align*}
But $\pr(\mbox{no $u\!-\!v$ path of length 7 in } Q_p \mid \mbox{$u$, $v$ are $Q_p$-good})$ is a weighted average of such probabilities, so
\begin{align*}
\pr(&(\mbox{no $u\!-\!v$ path of length 7 in } Q_p) \land (\mbox{$u$, $v$ are $Q_p$-good}))\\
& \leq 
 \pr(\mbox{no $u\!-\!v$ path of length 7 in } Q_p \mid \mbox{$u$, $v$ are $Q_p$-good})
\;\; \leq \;\; 
 e^{-p^7 d^2/5}.
\end{align*}
Now, by a union bound, the probability that there is a pair of $Q_p$-good vertices at distance 3 in $Q$ which are not joined by a path of length 7 in $Q_p$ is at most
\[ 2^d d^3 e^{-p^7 d^2/5} =  2^{-\Omega(d^2)}.\]
Similarly, with failure probability $2^{-\Omega(d^2)}$, if $d_Q(u,v)=2$ then
wvhp there is a $u\!-\!v$ path of length 6, and if $d_Q(u,v)=1$ then wvhp there is a $u\!-\!v$ path of length 1 or~5.
\end{proof}

The second preliminary lemma is deterministic.  
\begin{lemma} \label{lem.onecomp}
Let $Q'$ be a (fixed) spanning subgraph of $Q$.  Suppose that  each vertex has a $Q'$-good neighbour in $Q$, and that for each pair $u, v$ of $Q'$-good vertices at distance at most 3 in $Q$ there is a $u-v$ path in $Q'$.
Then for each pair $u, v$ of  $Q'$-good vertices there is a $u-v$ path in $Q'$,
% (of length at most $7d$); 
and so all $Q'$-good vertices are in the same component of~$Q'$.
\end{lemma}
\begin{proof}
Let $u, v$ be $Q'$-good vertices at distance $t> 3$ in $Q$.  We must show that there is a $u-v$ path in $Q'$.  Let $u=x_0, x_1,\ldots, x_{t-1}, x_t=v$ be a $u-v$ path in $Q$ of length $t$.  
For each $i = 1,\ldots,t-1$, let $y_i$ be a $Q'$-good neighbour in $Q$ of $x_i$, 
%or be $x_i$ if it is good, 
where we choose $y_1=u$ and $y_{t-1}=v$.  
Then since $d_Q(y_i,y_{i+1}) \leq 3$ for each $i=1,\ldots,t-2$ there is a $y_i-y_{i+1}$ path in $Q'$. % (of length at most 7).
Hence there is a $u-v$ path in $Q'$.
% (of length at most $7(t-2)$).
\end{proof}

We may now deduce an upper bound for $L_2$ as required.
When applying this upper bound, we shall later typically set $\gamma=3$, so that failure probabilities will be negligibly small.
\begin{lemma}\label{lem.L2leqN}
Let $0<p<1/2$ and let $\gamma>0$. Then there is a constant $N$ such that
$\pr(L_2>N) = o(2^{-\gamma d})$.
\end{lemma}
\begin{proof}
By a Chernoff bound and a union bound,
\begin{eqnarray*}
&& \pr(\mbox{some vertex has no $Q_p$-good neighbour in $Q$})\\
& \leq & 2^d \, \pr(\Bin(d,p) < pd/2)^d
\;\; \leq \;\;
2^d \, e^{-(pd/8)\, d} \;\; = \;\; 2^{-\Omega(d^2)}.
\end{eqnarray*}
Let $A$ be the event that all $Q_p$-good vertices in $Q_p$ are in the same component.  From the above bound and the last two lemmas
\begin{equation} \label{eqn.A}
 \pr( \bar{A})=2^{-\Omega(d^2)}.
\end{equation}

Now let $N= \lfloor \frac{16(1+\gamma)}{p} \rfloor$. 
%\m{\Cb worth saying this?}
If some component of the fragment has size at least $N+1$, then also the giant component has size at least $N+1$.
Hence, if $L_2 > N$ and the event $A$ holds then there is a component with size at least $N+1$ consisting entirely of bad vertices, and so in $Q_p$ there is a subtree with $N+1$ vertices each of which is bad.  But consider any subtree $T$ of $Q$ with $N+1$ vertices.  Since $Q$ is bipartite there is a set $W$ of at least $(N+1)/2$ vertices of $T$ which forms a stable set in $Q$; and the probability that each vertex in such a set $W$ is bad is
\[ \pr(\Bin(d,p)<pd/2)^{|W|} \; \leq \; e^{-\tfrac{pd}{8} \tfrac{N+1}{2}} \; \leq \; e^{-(1+\gamma)d} \]
% = e^{-(N+1)pd/16}\] 
by a Chernoff bound and the inequality $(N+1)pd/16 \geq (1+ \gamma)d$.  Hence 
by Lemma~\ref{lem.fewtrees} and a union bound, the probability that there is a subtree of $Q_p$ with $N+1$ vertices each of which is bad is at most
\[ 2^d (ed)^N e^{-(1+\gamma)d} = 
  (ed)^N (2/e)^{(1+\gamma)d} \,  2^{-\gamma d} = o(2^{-\gamma d}). \]
Finally, using also~(\ref{eqn.A}), we have
%\m{\CC now using $\land$ not $\cap$}
\[ \pr(L_2>N) \leq \pr((L_2>N) \land A) + \pr(\bar{A}) = o(2^{-\gamma d}),\]
which completes the proof.
\end{proof}

%%%%%%%%%%%%%%%%%%%%%%%%%%%%%%%%%%
%%%%%%%%%%%%%%%%%%%%%%%%%%%%%%%%%%

\section{Proofs of Theorems~\ref{thm.C1},~\ref{thm.C2} and~\ref{thm.C3}} %\ref{3t:1}} 
\label{sec.lastproofs}

In this section, we complete the proofs of Theorems~\ref{thm.C1},~\ref{thm.C2} and~\ref{thm.C3}.

\subsection{Proof of Theorem~\ref{thm.C1}}

We have already noted that part (a) of Theorem~\ref{thm.C1} will follow directly from Theorem~\ref{thm.C3} and inequality~(\ref{eqn.potail}). 

\begin{proof}[Proof of Theorem~\ref{thm.C1} part (b)]
Let $N$ be as in Lemma~\ref{lem.L2leqN} for $\gamma=3$, so that $\pr(L_2 >N) =o(2^{-3d})$. 
Consider an integer $t$ with ${\M}<t \leq N$.  By Markov's inequality and Lemma~\ref{3lem:Xi}, 
\[ \Pro(X_t \geq 1) \leq \Ex[X_t]=O(d^{t-1} (2q^t)^d)=e^{-\Omega(d)}, \]
where the last step follows since $2q^t<1$.
Hence wvhp the fragment $\cZ$ has no component containing exactly $t$ vertices.  Putting these results together, we see that $L_2 \leq m_p$ wvhp; and that
\[ \Ex[L_2] \leq \M + N \, \pr(\M < L_2 \leq N) + 2^d \, \pr(L_2 > N) = \M + e^{-\Omega(d)}.\]
But $L_2 \geq m_p$ wvhp by Theorem~\ref{cor.normal} part~(b) with $t=\M$ (since $X_t \geq \mu_t/2$ wvhp).
Hence $L_2=\M$ wvhp. It follows that $\Ex[L_2] \geq \M - e^{-\Omega(d)}$, and thus
$| \Ex[L_2] - \M | = e^{-\Omega(d)}$.

Now consider $\Var(L_2)$, starting with an upper bound.
We have  
\[ \Ex[(L_2-\M)^2 {\mathbb I}_{L_2 \leq N}] \leq N^2 \, \pr(L_2 \neq \M) = e^{-\Omega(d)},\]
and
\[ \Ex[(L_2-\M)^2 {\mathbb I}_{L_2 > N}] \leq 2^{2d} \, \pr(L_2 >N) = e^{-\Omega(d)},\]
where $\mathbb I$ denotes an indicator variable (as earlier).
Hence
\[ \Var(L_2) \leq \Ex[(L_2-\M)^2] = e^{-\Omega(d)},\]
which is an upper bound as required. 
Finally we show that
\begin{equation} \label{eqn.varL2}
  \Var(L_2) \gg q^d . %(q/2)^d.
\end{equation}
We start by noting a simple general lower bound on variance.
Let the random variable $L$ be integer-valued; let $k$ be an integer and let $x>0$; and suppose that both $\pr(L \leq k)$ and $\pr(L \geq k+1)$ are at least $x$.  Then $\Var(L) \geq x(1-x)$.

We know that $L_2= m_p$ wvhp.
Recall from Remark~\ref{remark} that in Theorem~\ref{cor.normal} both part (a) and the first half of part (c) hold for any given positive integer~$t$ (not just for $t \leq m_p$).  Let $t=m_p +1 \, ( \geq 2)$.
By the first half of part (c) of Theorem~\ref{cor.normal}
\[ \pr(L_2 \geq t) \geq \pr(X_t \geq 1) = \pr(\Po(\mu_t) \geq 1) + O(d^t q^{td}).\]
But since $\mu_t=o(1)$ and $2q^{m_p} \geq 1$, by part (a) of Theorem~\ref{cor.normal}
\[\pr(\Po(\mu_t) \geq 1) = (1+o(1))\, \mu_t 
= \Theta(d^{t-1} (2q^t)^d)  \gg d^t q^{td}.
\]
Thus
\[ \pr(L_2 \geq m_p +1) \geq (1+o(1))\, \mu_{m_p +1} \gg q^d. \] 
Now~(\ref{eqn.varL2}) follows from the general lower bound on variance given above, and this completes the proof of the theorem.
\end{proof}

%%%%%%%%%%%%%%%%%%%%%%%%%%

\subsection{Proof of Theorem~\ref{thm.C2}} %(\ref{3tp:2b})}
\label{3sec:proofd}

We prove the two parts of the theorem separately.  We denote the $r$-ball $B_r({\bf 0})$ centred on the vertex $\bf 0$ by $B_r$ for short.

\begin{proof}[Proof of Theorem~\ref{thm.C2} part (a)]
Let $s=m_p+1$ and let $V=V(Q)$. Recall from Theorem~\ref{thm.C1}(b) that $L_2 \leq \M$ wvhp.  We use $\deg(v)$ for the degree of a vertex $v$ in $Q_p$.  Also, for $v \in V$ and $W \subseteq V$, let $e(v,W)$ be the number of edges in $Q_p$ between $v$ and~$W$. For each subset $S \subseteq V$ with $|S|=s$ we have 
\begin{eqnarray*} 
\pr((S \subseteq V(\cZ)) \land (L_2 \leq \M))
& \leq & \pr(\deg(v) \leq \M -1 \;\; \forall v \in S)  \\ 
& \leq &
\pr(e(v,V \setminus S) \leq \M -1 \;\; \forall v \in S) \\ 
& = &
\left(\pr(\Bin(d-s,p) \leq \M -1)\right)^s\\ 
& \leq &
\left( \binom{d-s}{\M -1} q^{d-s-(\M -1)} \right)^s\\
& \leq &
\left( d^{\M -1} q^{d-2\M} \right)^s \; \leq \; (d/q^2)^{\M s} q^{sd}.
\end{eqnarray*}
Hence, for any $r>0$, % and any vertex $v \in V$,
%with $\cup_{S \subseteq B_r(u)}$ meaning that that also $|S|=s$, 
\begin{eqnarray} \label{eqn.probbound}
\pr(|V(\cZ)&\cap& B_r(u)| \geq s \mbox{ for some } u \in V ) \nonumber \\
& = &
\pr \Big( \bigcup_{u \in V} \bigcup_{S \subseteq B_r(u), |S|=s} (S \subseteq V(\cZ)) \Big) \nonumber \\
& \leq &
\pr \Big( \bigcup_{u \in V} \bigcup_{S \subseteq B_r(u), |S|=s} (S \subseteq V(\cZ)) \land (L_2 \leq \M) \Big) + \pr(L_2>\M) \nonumber \\
& \leq &
2^d \binom{|B_r|}{s} (d/q^2)^{\M s} q^{sd} + \pr(L_2>\M) \nonumber \\
& \leq &
 (d/q^2)^{\M s} \, |B_r|^s  \, (2 q^{s})^d + \pr(L_2>\M).
\end{eqnarray}
Since $s>\M$ and $q<1/2$, we have $2q^s<1$ and $1>\log_2 (1/q)- 1/{s} >0$.
Let $\eta_1$ be the unique $x \in (0,\frac12)$ such that $h(x)= \log_2 (1/{q}) - 1/{s}$. 
%\m{\CC Suppress $\eta_1$?}
% if no success with question at Theorem 2(a)?}
 Let $0< \eta<\eta_1$.  Then $h(\eta)< \log_2 (1/{q}) - 1/{s}$, and so
\[ 2 \, (2^{h(\eta)} q)^s < 1. \]
Set $r= \eta d$.  Then $|B_r|=2^{h(\eta)d+o(d)}$ by standard estimates.
Thus, by the last inequality, %~(\ref{eqn.eta1}),
\[ |B_r|^s \, (2q^s)^d = ( 2 \, (2^{h(\eta)} q)^s)^d 2^{o(d)}
= 2^{-\Omega(d)}.\]
Hence, by~(\ref{eqn.probbound}) and using $\pr(L_2>\M) = 2^{-\Omega(d)}$, we have
\[ \pr(|V(\cZ) \cap B_r(u)| \geq s \mbox{ for some } u \in V ) = 2^{-\Omega(d)} \]
as required.
\end{proof}
Consider $\eta_1$ in the above proof: it can be shown that if $\eta>\eta_1$ %and $r=\eta d$ 
then the expected number of $\eta d$-balls containing more than $\M$ vertices in $\cZ$ tends to $\infty$ as $d \to \infty$. % so our simple first moment proof will not work,
%\m{\CC worth some such comment?}

\begin{proof}[Proof of Theorem~\ref{thm.C2} part (b)]
Recall that $\eta^*$ is defined immediately before Theorem~\ref{thm.C2}.
We may assume that $\eps>0$ is sufficiently small that $\eta^*-\eps>0$ and $\eta^*+\eps<\tfrac12$.  Given $0<\eta \leq \tfrac12$, we have $|B_{\eta d}| = 2^{ h(\eta) d +o(d)}$, as we noted above. 
Also, $2^{-h(\eta^*)} =q$. Hence, by Theorem~1 (a), wvhp
\begin{eqnarray*}
  |B_{(\eta^* - \eps) d}| \cdot Z & \leq &
  2^{ h(\eta^*-\eps) d +o(d)} \cdot 2 \mu_1\\
  &=&
  2^{ (h(\eta^*-\eps) - h(\eta^*)+o(1)) d} \cdot 2^d \\
  &=&  2^{-\Omega(d)} \cdot 2^d,
\end{eqnarray*}
As the number of vertices within distance at most $(\eta^*-\eps)d$ of $\cZ$ is at most $|B_{(\eta^* - \eps) d}| \cdot Z$, this proves the first half of part (b).

For the second half, let $B'$ denote $B_{(\eta^* +\eps)d}$. By the definition of $\eta^*$, and recalling that $h(\eta)$ is strictly increasing on $(0,\tfrac12)$, we have $q^d |B'| = e^{\Omega(d)}$.  Since $Q^d$ is bipartite, there is a stable subset $B''$ of $B'$ with $|B''| \geq \tfrac12 |B'|$; and the probability that no vertex of $\cZ$ is in $B'$ is at most the probability that no vertex in $B''$ is isolated, which equals
\[  (1-q^d)^{|B''|} \leq \exp (-\tfrac12 q^d |B'|) = \exp(-e^{\Omega(d)}).\]
This bound refers to the ball $B'$ centred at $\bf 0$, and indeed to any fixed centre vertex. Taking a union bound over all $2^d$ possible centre vertices shows that the probability that some vertex is not within distance $(\eta^* +\eps)d$ of $\cZ$ is $\exp(-e^{\Omega(d)})$, and thus completes the proof.
\end{proof}

In the last part of the proof above, the number of isolated vertices in $B''$ has distribution $\Bin(|B''|,q^d)$, with mean at least $\tfrac12 |B'| q^d=e^{\Omega(d)}$.  Hence, by a Chernoff bound, the probability that there are at most $\tfrac14 |B'| q^d$ isolated vertices in the ball $B'$ is at most $e^{-e^{\Omega(d)}}$; and so, by a union bound, wvhp each $(\eta^* +\eps)d$-ball contains exponentially many isolated vertices.

%%%%%%%%%%%%%%%%%%%%%%%%%%%%%%%%%

\subsection{Proof of Theorem~\ref{thm.C3}}

By Lemma~\ref{lem.L2leqN} we may choose a fixed integer $N \geq 2$ such that $\pr(L_2>N) \leq 2^{-3d}$. 
\smallskip 

\noindent
\begin{proof}[Proof of Theorem~\ref{thm.C3} part (a)]
Note that $Z \leq 2^d$ and so
\[ Z \leq \sum_{t=1}^N X_t + 2^d {\mathbb I}_{L_2>N}.\]
By Lemma~\ref{3lem:Xi}, for each $2 \leq t \leq N$, $\mu_t = \Ex[X_t]= \Theta(d^{t-1}(2q^t)^d)$, so $\mu_t$ is $O(d(2q^2)^d)$. 
 Hence,
\begin{eqnarray*}
\Ex[Z] & \leq &
\sum_{t=1}^N \mu_t + 2^d\, \pr(L_2>N)\\
 & \leq &
 \mu_1  +O(d(2q^2)^d) + 2^{-2d} \; = \;
   (1+O(dq^d))  \mu_1. 
\end{eqnarray*} %\m{\Cb extra bracket deleted}
Also, of course, $\mu_1 + \mu_2 \leq \Ex[X] \leq \Ex[Z]$, which completes the proof for the expected values.
\medskip

Now consider variances. Let $X_{\leq N} = \sum_{t=1}^{N} X_t$ be the total number of components in $Q_p$ of size at most $N$; and similarly let $Z_{\leq N}=\sum_{t=1}^{N} tX_t$ be the total size of the components of size at most $N$.  Then
\[ \Var(Y) -  \Var(Y_{\leq N}) \leq \Ex[Y^2\!-\!Y_{\leq N}^2] \leq 2^{2d} \pr(L_2 > N) \leq 2^{-d},\]
and
\[  \Var(Y_{\leq N})\!-\!\Var(Y) \leq \Ex[Y\!\!+\!Y_{\leq N}] \, \Ex[Y\!\!-\!Y_{\leq N}] \leq 2 \Ex[Y] 2^d \pr(L_2\!>\!N)=o(2^{-d}),
\]
and so
\[ |\Var(Y) - \Var(Y_{\leq N})| = O(2^{-d}).\]
Hence by Lemma~\ref{lem.new9a}(b) and (c), % parts (b) and (c), 
with $H_1,\ldots,H_r$ listing a representative of each ambient-isomorphism class of connected cube subgraphs with at most $N$ vertices, we see that $\Var(Y) = (1+O(dq^d)) \mu_1$, as required.  
\end{proof}

\begin{proof}[Proof of Theorem~\ref{thm.C3} part (b)]
%Let us prove the case $Y=Z$ -- the case $Y=X$ is similar (and a little easier).
%\m{change to $Y$}
Let us show first that
\begin{equation} \label{eqn.dTVZ}
  d_{TV}(Y, \Po(\lambda)) = O(dq^d).
\end{equation}
Write $\lambda_{\leq N}$ for $\Ex[Y_{\leq N}]$. Now $d_{TV}(Y, \Po(\lambda))$ is at most
\[ %d_{TV}(Y, \Po(\lambda)) \leq 
d_{TV}(Y, Y_{\leq N}) + d_{TV} (Y_{\leq N}, \Po(\lambda_{\leq N})) + d_{TV}(\Po(\lambda_{\leq N}), \Po(\lambda)). \]
We consider the three terms in the sum in order. Firstly, we have
\[ d_{TV}(Y, Y_{\leq N}) \leq \pr(Y \neq Y_{\leq N}) = \pr(L_2>N) \leq 2^{-3d} = o(q^d).\]
Secondly, by Lemma~\ref{lem.new9a}(b) and (c) (with $H_i$ as above)
  \[ d_{TV} (Y_{\leq N}, \Po(\lambda_{\leq N})) = O(dq^d). \]
Thirdly, for $\mu, \delta >0$ the sum of independent $\Po(\mu)$ and $\Po(\delta)$ random variables has distribution $\Po(\mu+\delta)$; and so
\[ d_{TV}(\Po(\mu), \Po(\mu+\delta)) \leq \pr(\Po(\delta) \neq 0) = 1-e^{-\delta} \leq \delta.\]
 Thus
\[ d_{TV}(\Po(\lambda_{\leq N}), \Po(\lambda)) \leq
\lambda - \lambda_{\leq N} \leq 2^d \, \pr(L_2>N) \leq 2^{-2d} = o(q^d).\]
Putting these inequalities together we obtain~(\ref{eqn.dTVZ}).

Finally, since also $\Var(Z) \sim \lambda \to \infty$ as $d \to \infty$, it follows from~(\ref{eqn.dTVZ}) that $Z^*$ is asymptotically standard normal (see the discussion immediately before Theorem~\ref{thm.C3}).
This completes the proof of part (b), and thus of Theorem~\ref{thm.C3} (and thus also of Theorem~\ref{thm.C1}).
\end{proof}

%%%%%%%%%%%%%%%%%%%%%%%%%%%%%%%%%%
%%%%%%%%%%%%%%%%%%%%%%%%%%%%%%%%%%

\section{Joint distributions: proof of Theorem~\ref{thm.jointdistrib}} 
\label{sec.joint}
%\m{\CC introductory part moved earlier}

In this section we prove Theorem~\ref{thm.jointdistrib} on the joint distribution of the numbers of components of different types in the fragment.  We start by presenting a general lemma on approximating a joint distribution by a product of Poisson distributions. As in Subsection~\ref{subsec.prelimvar}, let $(A_i: i \in I)$ be a family of events with a dependency graph $L$,
and write $i \sim j$ if $i$ and $j$ are adjacent in $L$. For each~$i$, let $\pi_i= \pr(A_i)$ and let $\mathbb{I}_i$ be the indicator function of $A_i$.
Now we let $I$ be partitioned into $I_1 \cup \cdots \cup I_r$ for some $r \geq 1$.  For each $j \in [r]$, let $X_j=\sum_{i \in I_j} \ind_{A_i}$ and let $\lambda_j = \Ex[X_j]$. The following lemma is essentially a special case of Theorem 10.K of Barbour, Holst and Janson~\cite{BHJ} when all means $\lambda_j \to \infty$.  Sums and products over $j$ or $j'$ always mean over $j$ or $j'$ in $[r]$.
\begin{lemma} \label{lem.jointPo} 
With notation as above, assume that each $\lambda_j \to \infty$ as $d \to \infty$.  Then for $d$ sufficiently large 
\begin{eqnarray*}
&& d_{TV}({\cal L}(X_1,\ldots,X_r),\prod_{j} \Po(\lambda_j)) \\
& \leq &  \sum_j  \frac{\ln(\lambda_j)}{\lambda_j} \sum_{i \in I_j} \pi_i^2  + \sum_j\!\sum_{j'}\!\frac{\ln(\lambda_j \lambda_{j'})}{\sqrt{\lambda_j \lambda_{j'}}}\sum_{i \in I_j}\!\sum_{i' \in I_{j'}} \ind_{i \sim i'} ( \pr(A_{i}\!\land\!A_{i'}) + \pi_{i} \pi_{i'}).
\end{eqnarray*}
\end{lemma}
\begin{proof}[Proof of Theorem~\ref{thm.jointdistrib}]
As earlier, given $d$ let $\cS = \cS(d)$ be the set of subgraphs of $Q^d$ ambient isomorphic to one of the graphs $H_1,\ldots,H_r$. List the members of $\cS$
as $G_1,\ldots,G_N$; and let $A_i$ be the event that 
%the subgraph induced by the vertex set $V(G_i)$ is exactly 
$G_i$ is a component of $Q_p$. 
We let $i, i'$ run over $[N]$ and $j,j'$ run over $[r]$. For distinct $i,i'$ let $i \sim i'$ if either the vertex sets $V(G_i)$ and $V(G_{i'})$ intersect or there is an edge of $Q^d$ between them; and note that this gives
%  Observe that if $i \neq j$ and $i \not\sim j$ then the events $A_i$ and $A_j$ are independent, so we have 
a dependency graph $L$.
For each $j$, let $I_j=\{i : G_i \mbox{ is ambient isomorphic to } H_j\}$.

Now we can apply Lemma~\ref{lem.jointPo}. We must bound the two terms in the lemma. First, by~(\ref{eqn.pmax}), there is a constant $\alpha$ such that, for each~$j$,
\[ \sum_{i \in I_j} \pi_i^2 \leq \sum_{i \in I_j} \pi_i\cdot \alpha q^{v(G_i)d} = \lambda_j \cdot \alpha  q^{v(H_j)d}. \]
Hence
\begin{equation} \label{eqn.part1}
 \sum_{j} \frac{\ln(\lambda_j)}{\lambda_j} \sum_{i \in I_j} \pi_i^2 \leq
\alpha \sum_j \ln(\lambda_j) \, q^{v(H_j)d} = O(d q^{t^*d})
\end{equation}
since $\ln(\lambda_j)=O(d)$ uniformly over $j$.

For the second term, let $j, j' \in [r]$ (not necessarily distinct).  For $i \in I_j$ and $i' \in I_{j'}$, as in~(\ref{eqn.pmax2}) we have
\[ \pr(A_i \land A_{i'}) \leq \pi_i \pi_{i'}\, q^{-v(H_j)v(H_{j'})} \leq
\pi_i \, \alpha q^{v(H_{j'})d} q^{-v(H_j)v(H_{j'})}.\]
Hence, arguing as in the proof of~(\ref{eqn.D+}),
\begin{eqnarray*}
\beta(j,j') 
& := &
\sum_{i \in I_j} \sum_{i' \in I_{j'}} \ind_{i \sim i'} \big( \pr(A_i \land A_{i'}) + \pi_i \pi_{i'} \big)\\
& \leq &
\sum_{i \in I_j} \pi_i \cdot \alpha q^{v(H_{j'})d} \big(q^{-v(H_j)v(H_{j'})} +1 \big) v(H_j)d \, (ed)^{v(H_{j'})-1}\\
&=&
\lambda_j \cdot O\big( d^{v(H_{j'})} q^{v(H_{j'})d} \big) \; = \;
\lambda_j \cdot O\big( (d q^d)^{v(H_{j'})} \big).
\end{eqnarray*}
Similarly, swapping $j$ and $j'$, we have
\[ 
\beta(j,j')  \leq \lambda_{j'} \cdot O\big( (d q^d)^{v(H_{j})} \big);\]
and so
\[ 
\beta(j,j') \leq \sqrt{\lambda_j \lambda_{j'}} \cdot O\big( (d q^d)^{t^*_{j j'}} \big),\]
where $t^*_{j j'} = \min \{v(H_j), v(H_{j'})\}$.
Hence, 
\[ \frac{\ln(\lambda_j \lambda_{j'})}{\sqrt{\lambda_j \lambda_{j'}}} \,
\beta(j,j')  = O(d) \, O((dq^d)^{t^*_{j j'}}) = O\big(d^{t^*+1} q^{t^*d}\big). \]
So, summing over the bounded number of choices of $j$ and $j'$, we obtain
\begin{eqnarray*}
&&
 \sum_j\!\sum_{j'} \!\frac{\ln(\lambda_j \lambda_{j'})}{\sqrt{\lambda_j \lambda_{j'}}}\sum_{i \in I_j}\!\sum_{i' \in I_{j'}} \ind_{i \sim i'} ( \pr(A_{i}\!\land\!A_{i'}) + \pi_{i} \pi_{i'})\\
 & = &
  \sum_j\!\sum_{j'} \!\frac{\ln(\lambda_j \lambda_{j'})}{\sqrt{\lambda_j \lambda_{j'}}} \, \beta(j,j') \;\; = \;\;  
 O\big(d^{t^*+1} q^{t^*d}\big).
 \end{eqnarray*}
This result, together with~(\ref{eqn.part1}) lets us use Lemma~\ref{lem.jointPo} to complete the proof of Theorem~\ref{thm.jointdistrib}.
\end{proof}

%%%%%%%%%%%%%%%%%%%%%%%%%%%%%%%%
%%%%%%%%%%%%%%%%%%%%%%%%%%%%%%%%

\section{Concluding remarks} \label{sec.conclusion}

In Theorems~\ref{thm.C1} to~\ref{thm.jointdistrib} we have seen quite a full picture of the rich component structure of the random graph $Q_p=Q_p^d$, for fixed $p$ with $0<p<\frac 12$.
In particular, given an integer $t$ with $1\le t\le \M$, by Theorem~\ref{cor.normal} the number $X_t$ of components in $Q_p$ of size $t$, with mean $\mu_t$, has close to the Poisson distribution $\Po(\mu_t)$, and thus the standardised version $X_t^*$ has close to the standard normal distribution. In Proposition~\ref{prop.local} we gave a partial corresponding local limit result for convergence to the Poisson distribution: it would be interesting to learn more on such local behaviour.

It would also be interesting to consider the component structure in the case when $p$ is not fixed in $(0,\frac12)$, but $p=p(d)$ decreases suitably slowly to 0 as $d \to \infty$.  (Thanks to Remco van der Hofstadt for asking about this case.) 
\bigskip

\noindent {\bf Acknowledgement:}
We would like to thank the referees for their careful reading and very helpful comments.

%It would also be interesting to learn about the joint behaviour of $X_1,\ldots,X_{\M}$.
%\m{\CC joint distribs: is this `sketch proof' sufficient?  I started to write something at the end of section 2.2 (commented out) but it would take some time}

%What about the joint distribution of $X_1,\ldots,X_{\M}$? We can in fact handle this in much the same way as we handled the distribution of a single $X_t$, but based on Theorem 10.K of Barbour, Holst and Janson~\cite{BHJ} rather than on the results in Section~\ref{subsec.prelimvar}.  
%
%Write ${\cal L}(X_1,\ldots,X_{\M})$ for the joint law of $X_1,\ldots,X_{\M}$; and write $\prod_{j=1}^{\M} \Po(\mu_j)$ for the distribution of independent random variables $\Po(\mu_j)$.  Then we may see that
%\begin{equation}\label{eqn.jointXt}  d_{TV}\big({\cal L}(X_1,\ldots,X_{\M}),\prod_{j =1}^{\M} \Po(\mu_j)\big) =O(dq^d).\end{equation}
%Indeed, we have the following more detailed result in the spirit of Theorem~\ref{thm.C5} (which implies~(\ref{eqn.jointXt})).  Let $H_1,\ldots,H_r$ list all the canonical cube subgraphs with at most $\M$ vertices, and let $Y_i$ be the random number of ambient isomorphic copies of $H_i$ in $Q_p^d$, with mean $\lambda_i$.  Then
%\begin{equation}\label{eqn.jointYi} d_{TV}\big({\cal L}(Y_1,\ldots,Y_{r}),\prod_{i=1}^{r} \Po(\lambda_i)\big) =O(dq^d).\end{equation}

%
\addcontentsline{toc}{section}{Bibliography}
%\bibliographystyle{plain}
%\bibliography{Bibliography,C:/Users/Paul/Documents/Oxford/References/Combrefs,C:/Users/Paul/Documents/Oxford/References/Bookrefs}
%

\end{document}